\documentclass[11pt, reqno]{amsart}
\usepackage{latexsym, amsmath, amssymb, amsthm, verbatim}
\usepackage{cases}
\usepackage[colorlinks]{hyperref}
\usepackage{hypernat}
\usepackage{xcolor} 
\usepackage{color}
\usepackage{graphicx}

\usepackage[margin=1.5in]{geometry}

\newtheorem{thm}{Theorem}

\newtheorem{prop}[thm]{Proposition}

\theoremstyle{remark}
\newtheorem{rmk}[thm]{Remark}

\theoremstyle{definition}

\numberwithin{thm}{section}
\numberwithin{equation}{section}

\title[Flow of Legendre curves]{\protect{Inverse curvature flow of closed Legendre curves}}

\author{Takashi Kagaya}
\address{Takashi Kagaya, Muroran Institute of Technology, Japan}
\email{kagaya@muroran-it.ac.jp}

\author{Masatomo Takahashi}
\address{Masatomo Takahashi, Muroran Institute of Technology, Japan}
\email{masatomo@muroran-it.ac.jp}

\begin{document}

\begin{abstract}
In this paper, we deal with an inverse curvature flow of $\ell$-convex Legendre curves. 
Since the Legendre curve is a natural generalization of regular curve, the flow is a generalization of the classical inverse curvature flow of regular curves. 
For the initial value problem, we study on the unique existence of the flow in global time, the monotonicity of the number of the singular cusps with respect to $t>0$ and the asymptotic behavior of the flow as $t \to \infty$. 
Regarding the asymptotics, the flow asymptotically converges to one of the self similar solutions by scaling appropriately, and the convergence is completely categorized depending on the initial curve.
\end{abstract}


\maketitle

\section{Introduction}

The flow of plane curves by functions of their curvatures have been studied. 
Typical examples of the flow are the curvature flow, or it is also so-called the curve shortening flow (cf.\ \cite{An2, GaHa, Gr} and their citations), and the inverse curvature flow (cf.\ \cite{Ge, Kr, Ur} and their citations).
Research on these flows has primarily advanced within the frameworks of regular curve (see above references), geometric measure theory (cf.\ \cite{Br} and its citations), level set (cf.\ \cite{CGG, EvSp} and their citations) and etc. 
In particular, for the curvature and inverse curvature flow, solutions starting from sufficiently smooth closed initial curve can be solved as a family of smooth regular curves due to the smoothing effect of the equations. 
The questions addressed in this paper is: what types of flow preserve ``cusp type'' singularities allowing the initial curve has singular points? 
In this paper, to answer this question, we introduce a different framework of curves from the above and present results on the unique existence theory in global time and the asymptotic behavior analysis as $t \to \infty$ for the inverse curvature flow. 
Although, for the inverse curvature flow, its framework of curves has already been introduced and studied by Li and Wang \cite{LiWa}, our results constitute an expansion of their results. 
The difference between their results and our results will be stated later. 

The Legendre curve is known as a framework of curves that allows singularities such as ``cusps.''
Therefore, we introduce the framework. 
For the convenience of the reader, we recall some notions related to the Legendre curve and refer to \cite{FuTa1, FuTa2} for more details. 
An one dimensional map $X=X(u): I \to \mathbb{R}^2$ is called a frontal if there exists a unit vector field $\nu: I \to \mathbb{S}^1$ such that 
\[ \langle \partial_u X(u), \nu(u) \rangle = 0 \quad \text{for} \; \; u \in I, \]
where $\langle \cdot, \cdot \rangle$ is the inner product on $\mathbb{R}^2$. 
In this case, the pair $(X, \nu): I \to \mathbb{R}^2 \times \mathbb{S}^1$ is called a Legendre curve. 
Let $J$ be the anticlockwise rotation operator by $\pi/2$ for 2-dimensional vectors and define $\mu := J\nu$. 
The pair $\{\nu, \mu\}$ is called a moving frame of $X$. 
Then, the Legendre curvature $(\ell, \beta)$ of the Legendre curve can be defined as 
\[ \ell := \langle \partial_u \nu, \mu \rangle, \quad \beta := \langle \partial_u X, \mu \rangle. \]
We note that the Legendre curve may have singular points and $X$ have a singular point at zero points of $\beta$. 
For example, due to $\partial_u X = \beta \mu$, if the sign of $\beta$ is changed at some point, then the direction of the parametrization $X$ and the tangent vector $\mu$ is changed at the point (see Figure \ref{figure-cusp}).  
Moreover, $\ell/|\beta|$ coincides with the classical curvature $\kappa$ at the regular points $\beta \neq 0$. 
If $\beta(u) > 0$ or $\beta(u) < 0$ for any $u \in I$, then the image of $X$ can be parametrized as a regular curve. 
The Legendre curve $(X, \nu)$ is said $\ell$-convex when $\ell(u) > 0$ or $\ell(u) < 0$ for any $u \in I$. 
Since $\ell$-convex Legendre curves can be re-parametrized so that $\ell(u)>0$ for any $u \in I$, we use the term $\ell$-convex to mean that the Legendre curves has positive $\ell$.
Let $I$ hereafter be the periodic interval with periodic $2\pi$ so that $X$ represents a closed curve, and the periodic interval will be denoted by $\mathbb{S}^1_{2\pi}$. 

\begin{figure}[t]
\centering
\includegraphics[width=4.5cm]{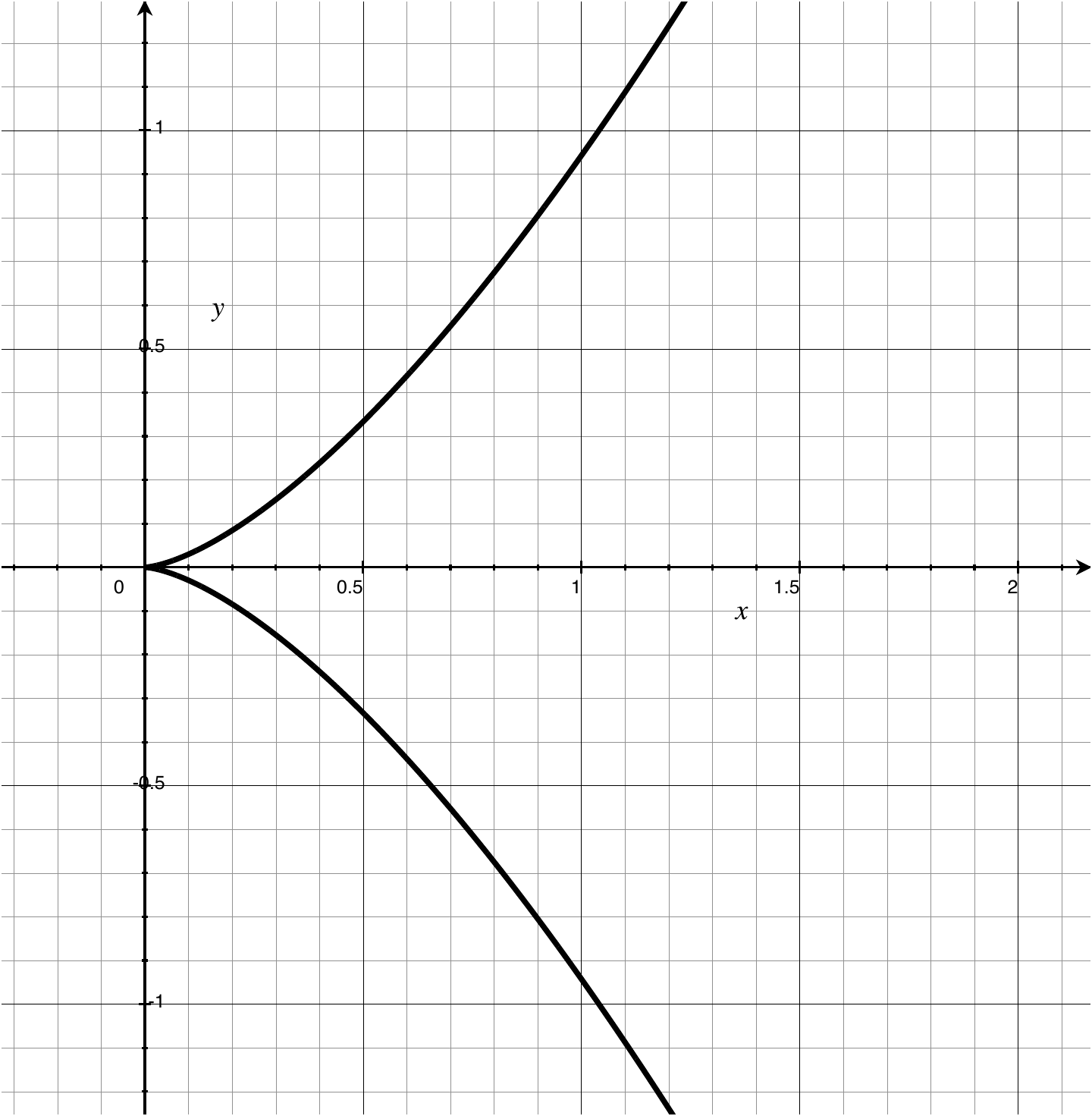}
\caption{The figure shows the image of $X$ for $(2,3)$-cusp $(X, \nu)$ defined by 
$X(u) = \begin{pmatrix}
u^2/2 \\
u^3/3
\end{pmatrix}$ and $\nu = \dfrac{1}{\sqrt{1+u^2}}\begin{pmatrix}
u \\
-1
\end{pmatrix}$. 
The Legendre curvature $(\ell, \beta)$ of this Legendre curve is $\ell(u) = 1/(1+u^2)^{3/2}, \beta(u) = u$. 
Therefore, the Legendre curve is $\ell$-convex and $\beta$ has a sign change at $u=0$.}
\label{figure-cusp}
\end{figure}

For a family of Legendre curves $(X, \nu)=(X(u,t), \nu(u,t)): \mathbb{S}^1_{2\pi} \times [0, \infty) \to \mathbb{R}^2$, we define the normal velocity $N$ and the tangent velocity $T$ as 
\[ N = \langle \partial_t X, \nu \rangle, \quad T = \langle \partial_t X, \mu \rangle. \] 
The family of Legendre curves is called an inverse curvature flow if $(X(\cdot, t), \nu(\cdot, t))$ is $\ell$-convex for any $t>0$ and $(X, \nu)$ satisfies the inverse curvature flow equation $N = \beta/\ell$. 
Notice that, if the Legendre curvature $(\ell, \beta)$ of $(X, \nu)$ satisfies $\beta \neq 0$ on $\mathbb{S}^1_{2\pi} \times (0,\infty)$, then the inverse curvature flow of Legendre curves coincides with the inverse curvature flow of regular curves. 
Our purpose is to prove the unique existence of solution in global time and to analyze the asymptotic behavior of it as $t \to \infty$. 

We first present the unique existence of the inverse curvature flow starting from given initial Legendre curve. 
In order to state the uniqueness of the flow, we introduce an equivalence for Legendre curves as follows: 
We say that two Legendre curves $(X_1, \nu_1), (X_2, \nu_2): \mathbb{S}^1_{2\pi} \to \mathbb{R}^2 \times \mathbb{S}^1$ are equivalent if and only if there exists a $C^1$-diffeomorphism $\phi: \mathbb{S}^1_{2\pi} \to \mathbb{S}^1_{2\pi}$ such that $X_2 = X_1 \circ \phi$. 
Therefore, for example, two Legendre curves $(X_1, \nu_1), (X_2, \nu_2)$ satisfying
\[ X_i(u) = \begin{pmatrix}
\cos(n_i u) \\
\sin(n_i u)
\end{pmatrix} \quad \text{for} \; \; u \in \mathbb{S}^1_{2\pi}, \; n_i \in \mathbb{N}, \; i=1, 2, \]
which describes a circle with $n_i$ revolution, are equivalent if $n_1 = n_2$ and inequivalent if $n_1 \neq n_2$. 
We also use the following functional spaces to state the regularity of the solution: 
Let $Y_1$ and $Y_2$ be norm spaces. 
For $\alpha > 0$, denote by $C(Y_1; Y_2), C^\alpha(Y_1; Y_2)$ and $C^\infty(Y_1; Y_2)$ the space of continuous, $C^\alpha$ and $C^\infty$ functions defined on $Y_1$ with values in $Y_2$, respectively. 
When $Y_2 = \mathbb{R}$, the range $Y_2$ is omitted in the above notions of the functional spaces (e.g.\ $C(Y_1)$). 
Then, the unique existence theory can be stated as follows: 

\begin{thm}\label{thm:initial problem} 
Let $\alpha \in (0, 1)$ and $(X_0, \nu_0) \in C^{1+\alpha} (\mathbb{S}^1_{2\pi}; \mathbb{R}^2 \times \mathbb{S}^1)$ be an $\ell$-convex Legendre curve such that $X_0$ does not describe a point. 
Then, an inverse curvature flow $(X, \nu) \in C([0, \infty); C^1 (\mathbb{S}^1_{2\pi}; \mathbb{R}^2 \times \mathbb{S}^1)) \cap C^\infty (\mathbb{S}^1_{2\pi} \times (0, \infty); \mathbb{R}^2 \times \mathbb{S}^1)$ starting from $(X_0, \nu_0)$ exists in global time and the flow is unique in the following sense: 
If $(\tilde{X}, \tilde{\nu}) \in C([0, \infty); C^1 (\mathbb{S}^1_{2\pi}; \mathbb{R}^2 \times \mathbb{S}^1)) \cap C^\infty (\mathbb{S}^1_{2\pi} \times (0, \infty); \mathbb{R}^2 \times \mathbb{S}^1)$ with $(\tilde{X}(\cdot, 0), \tilde{\nu}(\cdot, 0)) \in C^{1+\alpha}(\mathbb{S}^1_{2\pi}; \mathbb{R}^2 \times \mathbb{S}^1)$ is another inverse curvature flow such that $(\tilde{X}(\cdot, 0), \tilde{\nu}(\cdot, 0))$ is equivalent to $(X_0, \nu_0)$, then $(\tilde{X}(\cdot, t), \tilde{\nu}(\cdot, 0))$ is equivalent to $(X(\cdot, t), \nu(\cdot,t))$ for any $t > 0$. 
In particular, the inverse curvature flow $(X, \nu)$ is equivalent to a special inverse curvature flow $(\tilde{X}, \tilde{\nu}) \in C([0, \infty); C^1 (\mathbb{S}^1_{2\pi}; \mathbb{R}^2 \times \mathbb{S}^1)) \cap C^\infty (\mathbb{S}^1_{2\pi} \times (0, \infty); \mathbb{R}^2 \times \mathbb{S}^1)$ satisfying 
\begin{equation}\label{special inverse}
\partial_t \tilde{X} = \frac{\tilde{\beta}}{\tilde{\ell}} \tilde{\nu} + \frac{\partial_u \tilde{\beta}}{\tilde{\ell}^2} \tilde{\mu} \quad \text{on} \; \; \mathbb{S}^1_{2\pi} \times (0, \infty), 
\end{equation}
where $(\tilde{\ell}, \tilde{\beta})$ is the Legendre curvature of $(\tilde{X}, \tilde{\nu})$ and $\tilde{\mu} = J \tilde{\nu}$. 

Furthermore, letting $(\ell, \beta)$ be the Legendre curvature of $(X, \nu)$, then the zero points of $\beta(\cdot, t)$ are isolated for any $t>0$ and the number of the zero points $z(t)$ of $\beta(\cdot, t)$ is non-increasing with respect to $t > 0$. 
Moreover, if $\beta(u_0, t_0) = 0$ and $\partial_u \beta(u_0, t_0) = 0$ at $(u_0, t_0) \in \mathbb{S}^1_{2\pi} \times (0, \infty)$, then $z(t)$ is strictly decreasing at $t=t_0$, namely, $z(t_0 + \varepsilon) < z(t_0 - \delta)$ for any $\varepsilon > 0$ and $\delta > 0$. 
\end{thm}

The properties in Theorem \ref{thm:initial problem} on the zero points of $\beta(\cdot, 0)$ yields that 
\begin{itemize}
\item $\beta$ has degenerate zero points only at finite times $t = t_0$;  
\item $\partial_u \beta \neq 0$ at the zero points of $\beta$ except the above degenerate zero points for $t > 0$. 
\end{itemize}
Therefore, the sign of $\beta$ is changed at the zero points of $\beta$ except the degenerate zero points. 
In particular, it is known that if a frontal has a point such that $\beta=0, \partial_u \beta \neq 0$ and $\ell \neq 0$, then the frontal is diffeomorphic to the $(2,3)$-cusp defined in Figure \ref{figure-cusp} near the singular point (cf.\ \cite{BrGa, FuTa3}). 
This type of singularity is called a singular point of $(2,3)$-cusp. 
Theorem \ref{thm:initial problem} also shows that the number of $(2,3)$-cusps of $(X(\cdot, t), \nu(\cdot, t))$ is non-increasing. 
On the other hand, letting $(\ell_0, \beta_0)$ be the Legendre curvature of $(X_0, \nu_0)$, although $\beta_0$ is not zero function since the initial curve does not describe a point, $\beta_0$ can be $0$ on some interval in $\mathbb{S}^1_{2\pi}$. 
In this case, the initial curve may have singular corners other than cusps on the zero regions of $\beta_0$ (see \cite[Example 3.1]{FuTa2} for an example of Legendre curve such that $\beta$ has a zero region, which describes a corner of a square). 
Theorem \ref{thm:initial problem} also yields that the inverse curvature flow is regular for $t>0$ if $\beta_0(u) \ge 0$ or $\beta_0(u) \le 0$ for $u \in \mathbb{S}^1_{2\pi}$.

Due to the equivalence in Theorem \ref{thm:initial problem}, it is sufficient to study the asymptotic behavior of the special inverse curvature flow satisfying \eqref{special inverse} to obtain the asymptotics of general inverse curvature flow. 
In order to state the asymptotic behavior result, we need to introduce self similar solutions for the inverse curvature flow. 

An inverse curvature flow $(X, \nu)$ is called a self similar solution if $X$ and $\nu$ are formed by 
\[ X(u, t) = \lambda^*(t) X^*(u), \quad \nu(u, t) = \nu^*(u) \quad \text{for} \; \; (u, t) \in \mathbb{S}^1_{2\pi} \times [0, \infty) \]
for some $\ell$-convex Legendre curve $(X^*, \nu^*): \mathbb{S}^1_{2\pi} \to \mathbb{R}^2 \times \mathbb{S}^1$ and some function $\lambda^*: [0, \infty) \to \mathbb{R}$. 
Since an $\ell$-convex Legendre curve $(X, \nu): \mathbb{S}^1_{2\pi} \to \mathbb{R}^2 \times \mathbb{S}^1$ can be re-parametrized so that 
\begin{equation}\label{n-rotation} 
\nu(u) = \begin{pmatrix}
\sin(nu) \\
-\cos(nu)
\end{pmatrix} \quad \text{for} \; \; u \in \mathbb{S}^1_{2\pi} 
\end{equation}
for some $n \in \mathbb{N}$ (see Proposition \ref{prop:re-para} for the proof), the following theorem gives a classification of self similar solutions. 

\begin{thm}\label{thm:self similar}
An inverse curvature flow $(\lambda^*(t) X^*(u), \nu^*(u))$ is a self similar solution satisfying \eqref{n-rotation} replaced $\nu$ by $\nu^*$ for some $n \in \mathbb{N}$ if and only if 
\begin{align} 
&\lambda^*(t) = e^{(1-\frac{m^2}{n^2})t} \quad \text{for} \; \; t \ge 0, \label{self-similar-lambda}\\
&\begin{aligned}X^*(u) = 
& \; C_1 \begin{pmatrix}
\frac{n}{n^2-m^2} \sin(nu)\cos(mu) - \frac{m}{n^2-m^2} \cos(nu)\sin(mu) \\
-\frac{n}{n^2-m^2} \cos(nu)\cos(mu) - \frac{m}{n^2-m^2} \sin(nu)\sin(mu)
\end{pmatrix} \\
+&\; C_2 \begin{pmatrix}
\frac{n}{n^2-m^2} \sin(nu)\sin(mu) + \frac{m}{n^2-m^2} \cos(nu)\cos(mu) \\
-\frac{n}{n^2-m^2} \cos(nu)\sin(mu) + \frac{m}{n^2-m^2} \sin(nu)\cos(mu)
\end{pmatrix} \quad \text{for} \; \; u \in \mathbb{S}^1_{2\pi},
\end{aligned} \label{self-similar-X}
\end{align}
where $m, C_1, C_2$ are constants satisfying the either of the following conditions: 
\begin{itemize}
\item[(i)] $m \in \mathbb{N} \setminus \{n\}$ and $(C_1, C_2) \in \mathbb{R}^2 \setminus \{(0, 0)\}$. 
\item[(ii)] $m = 0, C_1 \neq 0$ and $C_2=0$. 
\end{itemize}
Furthermore, if the self similar solution $(\lambda^*(t) X^*(u), \nu^*(u))$ is represented by \eqref{self-similar-lambda} and \eqref{self-similar-X}, then the Legendre curvature $(l^*, \beta^*)$ of $(X^*, \nu^*)$ satisfies
\[ \ell^*(u) = n, \quad \beta^*(u) = C_1 \cos(mu) + C_2 \sin(mu) \quad \text{for} \; \; u \in \mathbb{S}^1_{2\pi}.  \]
\end{thm}
We note that, in the case (ii), $X^*$ can be re-written as 
\[ X^*(u) = \frac{C_1}{n} \begin{pmatrix}
\sin (nu) \\
-\cos (nu)
\end{pmatrix} \quad \text{for} \; \; u \in \mathbb{S}^1_{2\pi}, \]
and thus $X^*$ describes a circle with $n$ revolution. 
Another remark is the constants $C_1$ and $C_2$ do not affect to the number of the singular cusps of the ``image'' of $X^*$ (see Remark \ref{rmk:self similar} for details). 
Some figures of the image of $X^*$ are listed at the end of Section \ref{sec:self similar}. 
From the form of $\lambda^*$ in \eqref{self-similar-lambda}, the self similar solution shrinks to a point when $m > n$ and expands into the distance when $m < n$.
In order to state that the flow obtained in Theorem \ref{thm:initial problem} asymptotically converges to which self similar solution as $t \to \infty$, we denote by $(X^*_{n, m, C_1, C_2}, \nu^*_n)$ and $\lambda^*_{n, m}$ the profile Legendre curve $(X^*, \nu^*)$ and scaling function $\lambda^*$ which can be constructed from $n, m, C_1$ and $C_2$ as in Theorem \ref{thm:self similar}. 
We then present a convergence result of the inverse curvature flow as follows: 

\begin{thm}\label{thm:asymptotics}
Let $\alpha \in (0,1)$, $(X_0, \nu_0) \in C^{1+\alpha} (\mathbb{S}^1_{2\pi}; \mathbb{R}^2 \times \mathbb{S}^1)$ be an $\ell$-convex Legendre curve, such that $X_0$ does not describe a point, satisfying \eqref{n-rotation} for some $n \in \mathbb{N}$, and $(\ell_0, \beta_0)$ be the Legendre curvature of $(X_0, \nu_0)$. 
Let $(X, \nu)$ be the special inverse curvature flow satisfying \eqref{special inverse} and starting from $(X_0, \nu_0)$ obtained in Theorem \ref{thm:initial problem}. 
Let also $\{a_k\}_{k \in \mathbb{N} \cup \{0\}}$ and $\{b_k\}_{k \in \mathbb{N}}$ be respectively the family of the Fourier cosine coefficients and Fourier sine coefficients of $\beta_0$, namely, $a_0$, $a_k$ and $b_k$ are given by 
\begin{align*}
a_0 = \frac{1}{2\pi} \int_{\mathbb{S}^1_{2\pi}} \beta_0 (u) \; du, \quad a_k = \int_{\mathbb{S}^1_{2\pi}} \beta_0(u) \cos(ku) \; du, \quad b_k = \int_{\mathbb{S}^1_{2\pi}} \beta_0(u) \sin(ku) \; du
\end{align*}
for $k \in \mathbb{N}$ and $\beta_0$ can be decomposed by 
\[ \beta_0(u) = a_0 + \left(\sum_{k=1}^\infty a_k \cos(ku) + b_k \sin(ku) \right) \quad \text{for} \; \; u \in \mathbb{S}^1_{2\pi}. \]
Then, the following properties hold: 
\begin{itemize}
\item[(i)] It holds that $a_n = b_n = 0$. 
\item[(ii)] If $a_0 \neq 0$, then there exists $p \in \mathbb{R}^2$ such that 
\[ \frac{X(\cdot, t) - p}{\lambda^*_{n, 0}(t)} = \frac{X(\cdot, t) - p}{e^t} \to X^*_{n, 0, a_0, 0} \quad \text{in} \; \; C^\infty(\mathbb{S}^1_{2\pi}) \; \; \text{as} \; \; t \to \infty. \]
\item[(iii)] If $a_0 = 0$ and there exists $m \in \mathbb{N} \setminus \{n\}$ such that 
\begin{equation}\label{as-fourier-coe} 
a_k = b_k = 0 \quad \text{for} \; \; k=1, 2, \cdots, m-1, \quad (a_m, b_m) \neq (0,0), 
\end{equation}
then there exists $p \in \mathbb{R}^2$ such that 
\[ \frac{X(\cdot, t) - p}{\lambda_{n, m}^*(t)} = \frac{X(\cdot, t) - p}{e^{(1-\frac{m^2}{n^2})t}} \to X^*_{n, m, a_m, b_m} \quad \text{in} \; \; C^\infty(\mathbb{S}^1_{2\pi}) \; \; \text{as} \; \; t \to \infty. \]
\end{itemize}
\end{thm}

For the inverse curvature flow of embedded strictly convex closed regular curves, it was shown that the flow expands at the speed $e^t$ and converges to a circle when the flow is scaled at the rate of $e^{-t}$ (cf.\ \cite{Kr}). 
Therefore, Theorem \ref{thm:asymptotics} is an expansion of the result above. 
We note that Theorem \ref{thm:initial problem} and (iii) in Theorem \ref{thm:asymptotics} yield that, if $\beta_0$ satisfies $\int_{\mathbb{S}^1_{2\pi}} \beta_0 \; du = 0$, then the flow have $(2,3)$-cusps for any time $t > 0$.

As we mentioned, the inverse curvature flow of Legendre curves was already studied in \cite{LiWa}. 
Therefore, we here compare their results with ours. 
First different point is that they deal with the flow of Legendre curves which can be re-parametrized so that \eqref{n-rotation} holds only for $n=1$. 
Under this setting, they mainly study on a generalized isoperimetric type inequality for $\ell$-convex Legendre curves. 
As an application of it, they proved that an inverse curvature flow $(X, \nu)$ whose motion is governed by 
\[ \partial_t X = \frac{\beta}{\ell} \nu + \frac{1}{\ell} \partial_u\left(\frac{\beta}{\ell}\right) \mu \]
converges a self similar solution after a suitable rescaling. 
Since we deal with the inverse curvature flow of Legendre curves which can be re-parametrized so that \eqref{n-rotation} holds for general integer $n\in \mathbb{N}$, our results are expansion of their results from this perspective. 
Moreover, we precisely analyze the order of the scaling required to the convergence. 
Second different point is that we give a rigorous proof of the unique existence theory for the inverse curvature flow. 
In particular, as far as we know, the results that the flow is unique except for the diffeomorphic transformations when the normal velocity is specified by $N=\beta/\ell$, the inverse curvature flow only has $(2,3)$-cusps except degenerate zero points of $\beta$, and that the number of singularities is non-increasing are new. 

The remained part of this paper is organized as follows: 
In Section \ref{sec:geometric-eq}, we derive differential equations of geometric quantities for smooth flows of Legendre curves. 
In this section, the normal velocity is not restricted to $N=\beta/\ell$, and thus general flows are concerned. 
Furthermore, the $\ell$-convexity of the Legendre curves is not assumed. 
Theorem \ref{thm:initial problem} is proved in Section \ref{sec:existence} by using the geometric differential equations derived in Section \ref{sec:geometric-eq}. 
We first study on some properties related to the diffeomorphic transformations of the inverse curvature flow, which will be used to prove the uniqueness of the flow. 
The unique existence theory is proved by solving the differential equations of the Legendre curvature and constructing the flow of Legendre curves from its solution. 
In Section \ref{sec:self similar}, we prove Theorem \ref{thm:self similar}. 
By substituting the Legendre curvature of self similar solution into the geometric equations used in Section \ref{sec:existence}, we obtain an exact representation of the Legendre curvature, and we then complete the proof. 
Some figures of examples of the self similar solutions are also listed in this section. 
Theorem \ref{thm:asymptotics} is proved in Section \ref{sec:asymptotics} by applying Fouriier's method to the geometric differential equation of $\beta$ and using some functional inequalities to obtain decay estimates.

\subsection*{Acknowledgments} The work of the first author is supported by JSPS KAKENHI Grant No.\ JP23K12992, JP23H00085 and JP23K20802.
The work of the second author is supported by JSPS KAKENHI Grant  No.\ JP24K06728 and JP22KK0034.

\section{Geometric equations for a flow of Legendre curves} \label{sec:geometric-eq}

In this section, we derive geometric differential equations for a smooth flow of Legendre curves $(X, \nu) \in C^\infty(\mathbb{S}^1_{2\pi} \times (0, \infty); \mathbb{R}^2 \times \mathbb{S}^1)$. 
The velocity $\partial_t X$ can be decompose by the normal velocity $N$ and the tangent velocity $T$ as 
\begin{equation}\label{decompose-velocity}
\partial_t X(u,t) = N(u,t) \nu(u, t) + T(u, t) \mu(u, t) \quad \text{for} \; \; (u,t) \in \mathbb{S}^1_{2\pi} \times (0, \infty). 
\end{equation}
The purpose in this section is to derive differential equations of geometric quantities depending on $N$ and $T$. 
We note that the $\ell$-convexity of $(X, \nu)$ is not assumed in this section. 
The differential equations will be used choosing $N = \beta/\ell$ to study the inverse curvature flow later. 
The calculations in this section are based on a standard derivation of geometric differential equations for flows of regular curves (e.g.\ \cite{GaHa}).

Throughout this section, let $(X, \nu)$ be a smooth flow of Legendre curves and denote by $N, T$ and $(\ell, \beta)$ the normal velocity, the tangent velocity and the Legendre curvature of $(X, \nu)$, respectively. 
We further introduce an angle function $\Theta: \mathbb{S}^1_{2\pi} \times (0, \infty) \to \mathbb{S}^1_{2\pi}$ satisfying 
\[ \nu(u, t) = \begin{pmatrix}
\sin \Theta(u, t) \\
- \cos \Theta(u, t)
\end{pmatrix} \quad \text{for} \; \; (u, t) \in \mathbb{S}^1_{2\pi} \times (0, \infty). \]
We here identify the value of $\Theta$ modulus $2\pi$ so that $\Theta$ is smooth on $\mathbb{S}^1_{2\pi} \times (0, \infty)$. 
Notice that, in this setting, $\mu$ can be represented by 
\[ \mu(u, t) = \begin{pmatrix}
\cos \Theta(u, t) \\
\sin \Theta(u,t)
\end{pmatrix} \quad \text{for} \; \; (u, t) \in \mathbb{S}^1_{2\pi} \times (0, \infty). \]

We first derive a geometric differential equation of $\beta$ as follows:

\begin{prop}\label{prop:eq-beta}
The time derivative of $\beta$ satisfies 
\begin{equation}\label{eq-beta}
\partial_t \beta = N\ell + \partial_u T \quad \text{on} \; \; \mathbb{S}^1_{2\pi} \times (0, \infty). 
\end{equation}
\end{prop}

\begin{proof}
We first note that $\langle \mu, \partial_t \mu \rangle = 0$ can be obtained by differentiating $|\mu|^2 = 1$ with respect to $t$, which yields $\langle \partial_u X, \partial_t \mu \rangle = 0$. 
Therefore, the decomposition \eqref{decompose-velocity} and the Frenet-Serret type formula $\partial_u \nu = \ell \mu$ and $\partial_u \mu = - \ell \nu$ imply
\begin{align*}
\partial_t \beta =&\; \partial_t \langle \partial_u X, \mu \rangle = \langle \partial_u \partial_t X, \mu \rangle = \langle \partial_u(N\nu + T\mu), \mu \rangle \\
=&\; \langle (\partial_u N) \nu + N\ell \mu + (\partial_u T) \mu - T\ell \nu, \mu \rangle = N\ell + \partial_u T. 
\end{align*}
This derive the equation \eqref{eq-beta}. 
\end{proof}

We next derive geometric differential equations of the moving frame.

\begin{prop}
The time derivative of $\mu$ and $\nu$ satisfies 
\begin{align}
\beta \partial_t \mu = (\partial_u N - T\ell) \nu \quad \text{on} \; \; \mathbb{S}^1_{2\pi} \times (0, \infty), \label{eq-mu}\\
\beta \partial_t \nu = (T\ell - \partial_u N) \mu \quad \text{on} \; \; \mathbb{S}^1_{2\pi} \times (0, \infty). \label{eq-nu}
\end{align}
\end{prop}

\begin{proof}
Differentiating $\partial_u X = \beta \mu$ with respect to $t$, the equation \eqref{eq-beta} yields
\begin{equation}\label{eq-mu1}
\partial_t \partial_u X = (\partial_t \beta) \mu +  \beta \partial_t \mu = (N\ell + \partial_u T) \mu + \beta \partial_t \mu. 
\end{equation}
On the other hand, since $\partial_u$ and $\partial_t$ are commutative, we have by \eqref{decompose-velocity} and the Frenet-Serret type formula
\[ \partial_t \partial_u X = \partial_u (N \nu + T \mu) = (N\ell + \partial_u T) \mu + (\partial_u N - T\ell) \nu. \]
Combining it and \eqref{eq-mu1}, we obtain \eqref{eq-mu}. 
The equation \eqref{eq-nu} can be obtain by rotating the both side of \eqref{eq-mu} clockwise by $\pi/2$. 
\end{proof}

The following proposition present geometric differential equations of the angle function.

\begin{prop}
The derivatives of $\Theta$ satisfies 
\begin{align}
&\beta \partial_t \Theta = T\ell - \partial_u N \quad \text{on} \; \; \mathbb{S}^1_{2\pi} \times (0, \infty), \label{eq-theta-t}\\
&\partial_u \Theta = \ell \quad \text{on} \; \; \mathbb{S}^1_{2\pi} \times (0, \infty). \label{eq-theta-u}
\end{align}
\end{prop}

\begin{proof}
The first equation \eqref{eq-theta-t} follows from 
\[ \partial_t \mu = \partial_t \begin{pmatrix}
\cos \Theta \\
\sin \Theta
\end{pmatrix} =
\partial_t \Theta \begin{pmatrix}
- \sin \Theta \\
\cos \Theta
\end{pmatrix} = - (\partial_t \Theta) \nu \]
and the equation \eqref{eq-mu}. 
The remained one \eqref{eq-theta-u} also follows from $\partial_u \nu = (\partial_u \Theta) \mu$ and the definition of $\ell$.  
\end{proof}

We finally derive a geometric differential equation of $\ell$.

\begin{prop}\label{prop:eq-l}
The time derivative of $\ell$ satisfies 
\begin{equation}\label{eq-l}
\beta^2 \partial_t \ell = \beta \partial_u  (T\ell - \partial_u N) - \partial_u \beta (T\ell - \partial_u N) \quad \text{on} \; \; \mathbb{S}^1_{2\pi} \times (0, \infty). 
\end{equation}
\end{prop}

\begin{proof}
Due to \eqref{eq-theta-t} and \eqref{eq-theta-u}, we have 
\begin{align*}
\beta^2 \partial_t \ell = \beta^2 \partial_t \partial_u \Theta = \beta \partial_u (\beta \partial_t \Theta) -  \beta \partial_u \beta \partial_t \Theta  = \beta \partial_u (T\ell - \partial_u N) - \partial_u \beta (T\ell - \partial_u N), 
\end{align*}
which derive the equation \eqref{eq-l}. 
\end{proof}

\begin{rmk}
The equation \eqref{eq-l} can be changed to 
\begin{equation}\label{eq-l-changed}
\partial_t \ell = \partial_u \left( \frac{T\ell - \partial_u N}{\beta} \right) 
\end{equation}
except the zero points of $\beta$. 
\end{rmk} 

\section{Unique existence of the inverse curvature flow} \label{sec:existence}

In this section, we study the unique existence of the inverse curvature flow in global time as stated in Theorem \ref{thm:initial problem}. 
The proof is divided step by step. 
We first study on a sufficient condition of the tangent velocity of the inverse curvature flow. 
We then prove that any inverse curvature flow having the tangent velocity satisfying it can be diffeomorphic transformed into the special inverse curvature flow whose velocity formed by \eqref{special inverse}. 
Therefore, the uniqueness of the inverse curvature flow in the sense stated in Theorem \ref{thm:initial problem} can be obtained by proving the uniqueness of the flow $(X, \nu)$ satisfying \eqref{special inverse} and starting from $(X_0, \nu_0)$. 
We thus finally prove the statement in Theorem \ref{thm:initial problem} for the special inverse curvature flow.

We first present the sufficient condition of the tangent velocity of the inverse curvature flow. 

\begin{prop}\label{prop:tangent-velocity}
Let $(X, \nu) \in C^\infty(\mathbb{S}^1_{2\pi} \times (0, \infty); \mathbb{R}^2 \times \mathbb{S}^1)$ be a smooth inverse curvature flow. 
Let $T$ and $(\ell, \beta)$ be respectively the tangent velocity and the Legendre curvature of $(X, \nu)$. 
Then, there exists $F \in C^\infty(\mathbb{S}^1_{2\pi} \times (0, \infty))$ such that 
\begin{equation}\label{form-T} 
T(u, t) = \frac{\partial_u\beta(u, t)}{(\ell(u, t))^2} + F(u, t) \beta(u, t) \quad \text{for} \; \; (u, t) \in \mathbb{S}^1_{2\pi} \times (0, \infty). 
\end{equation} 
\end{prop}

\begin{proof}
Due to \eqref{eq-mu} and $N=\beta/\ell$, we have 
\begin{equation}\label{form-T1} 
T\nu = \frac{1}{\ell} \left\{\partial_u\left(\frac{\beta}{\ell} \right) \nu - \beta \partial_t \mu \right\}  = \frac{\partial_u \beta}{\ell^2} \nu - \frac{\beta \partial_u \ell}{\ell^3} \nu - \frac{\beta}{\ell} \partial_t \mu. 
\end{equation} 
Taking the inner product of both sides of \eqref{form-T1} with $\nu$ yields \eqref{form-T}, where 
\[ F = -\frac{\partial_u \ell}{\ell^3} - \frac{1}{\ell} \langle \partial_t \mu, \nu\rangle. \]
This completes the proof. 
\end{proof}

\begin{rmk}
For a general smooth flow of Legendre curves $(X, \nu)$, due to the equation \eqref{eq-mu}, the relation between the normal velocity $N$ and the tangent velocity $T$ can be written as 
\[ \ell(u, t) T(u, t) = \partial_u N (u, t) + F(u, t) \beta(u, t) \quad \text{for} \; \; (u, t) \in \mathbb{S}^1_{2\pi} \times (0, \infty) \]
for some smooth function $F \in C^\infty (\mathbb{S}^1_{2\pi} \times (0, \infty))$. 
\end{rmk}

We next prove that any inverse curvature flow having the tangent velocity satisfying \eqref{form-T} can be diffeomorphic transformed into the special inverse curvature flow. 

\begin{prop}\label{prop:re-para-time}
Let $\alpha \in (0, 1)$. 
For any smooth function $F \in C^\infty(\mathbb{S}^1_{2\pi} \times (0, \infty))$, let $(X, \nu) \in C([0, \infty); C^1(\mathbb{S}^1_{2\pi}; \mathbb{R}^2 \times \mathbb{S}^1)) \cap C^\infty(\mathbb{S}^1_{2\pi} \times (0, \infty); \mathbb{R}^2 \times \mathbb{S}^1)$ be an inverse curvature flow starting from $(X_0, \nu_0) \in C^{1+\alpha}(\mathbb{S}^1_{2\pi}; \mathbb{R}^2 \times \mathbb{S}^1)$ and satisfying \eqref{form-T}. 
Then, there exists an integer $n \in \mathbb{N}$ and a function $\phi \in  C([0, \infty); C^1(\mathbb{S}^1_{2\pi}; \mathbb{S}^1_{2\pi})) \cap C^\infty(\mathbb{S}^1 \times (0, \infty); \mathbb{S}^1)$ starting from a $C^1$-diffeomorphism $\phi_0 \in C^{1+\alpha} (\mathbb{S}^1_{2\pi}; \mathbb{S}^1_{2\pi})$ such that $\phi(\cdot, t): \mathbb{S}^1_{2\pi} \to \mathbb{S}^1_{2\pi}$ is a diffeomorphism for $t \ge 0$ and the re-parametrized flow $(\tilde{X}(u, t), \tilde{\nu}(u, t)) := (X(\phi(u,t), t), \nu(\phi(u, t), t))$ satisfies 
\begin{equation}\label{re-para-time} 
\tilde{\nu}(u, 0) = \begin{pmatrix}
\sin(nu) \\
-\cos(nu)
\end{pmatrix} \quad \text{for} \; \; u \in \mathbb{S}^1_{2\pi}, \quad \partial_t \tilde{X} = \frac{\tilde{\beta}}{\tilde{\ell}} \tilde{\nu} + \frac{\partial_u \tilde{\beta}}{\tilde{\ell}^2} \tilde{\mu} \quad \text{on} \; \; \mathbb{S}^1_{2\pi} \times (0, \infty), 
\end{equation}
where $\tilde{\mu}$ and $(\tilde{\ell}, \tilde{\beta})$ are respectively the unit tangent vector and the Legendre curvature of $(\tilde{X}, \tilde{\nu})$. 
\end{prop}

The proof of Proposition \ref{prop:re-para-time} will be completed by deriving an initial value problem for a partial differential equation of $\phi$ to satisfies \eqref{re-para-time} under the assumption $\partial_u \phi \neq 0$ on $\mathbb{S}^1_{2\pi} \times [0, \infty)$, and then solving the problem. 
Therefore, it also should be proved that a solution $\phi$ to the problem satisfies $\partial_u \phi \neq 0$ on $\mathbb{S}^1_{2\pi} \times [0, \infty)$ to complete the proof. 

The initial function $\phi_0$ will be chosen as the function obtained in Proposition \ref{prop:re-para} to re-parametrize $(X_0, \nu_0)$. 
We here derive the partial differential equation. 
We have by a simple calculation 
\begin{align*}
&\tilde{\ell}(u,t) = (\partial_u \phi(u, t)) \ell(\phi(u, t), t) \quad \text{for} \; \; (u, t) \in \mathbb{S}^1_{2\pi} \times (0, \infty), \\
&\tilde{\beta}(u, t) = (\partial_u \phi(u, t)) \beta(\phi(u, t), t) \quad \text{for} \; \; (u, t) \in \mathbb{S}^1_{2\pi} \times (0, \infty). 
\end{align*}
Therefore, since $(X, \nu)$ satisfies \eqref{form-T}, we obtain 
\begin{equation}\label{deri-PDE}
\begin{aligned}
\partial_t \tilde{X}(u, t) =&\;  (\partial_t \phi(u, t)) \partial_u X(\phi(u, t), t) + \partial_t X(\phi(u, t), t) \\
=&\; (\partial_t \phi(u, t)) \beta(\phi(u,t), t) \tilde{\mu}(u, t) + \frac{\tilde{\beta}(u, t)}{\tilde{\ell}(u,t)} \tilde{\nu}(u, t) + \frac{\partial_u \tilde{\beta}(u, t)}{(\tilde{\ell}(u, t))^2} \tilde{\mu}(u, t) \\
&\; + \left(F(\phi(u, t), t) - \frac{\partial_u^2 \phi(u, t)}{(\partial_u \phi(u, t))^2 (\ell(\phi(u, t), t))^2} \right) \beta(\phi(u, t), t) \tilde{\mu}(u, t)
\end{aligned}
\end{equation}
for $(u, t) \in \mathbb{S}^1_{2\pi} \times (0, \infty)$. 
It thus is sufficient that $\phi$ satisfies 
\begin{equation}\label{deri-initial-pro} 
\begin{cases}
\partial_t \phi(u,t) = \frac{\partial_u^2 \phi(u, t)}{(\partial_u \phi(u, t))^2 (\ell(\phi(u, t), t))^2} - F(\phi(u, t), t) & \text{for} \; \; (u, t) \in \mathbb{S}^1_{2\pi} \times (0, \infty), \\
\phi(u, 0) = \phi_0(u) & \text{for} \; \; u \in \mathbb{S}^1_{2\pi}. 
\end{cases}
\end{equation}
In order to solve \eqref{deri-initial-pro}, we need a priori gradient estimate as follows: 

\begin{prop}\label{prop:apriori-grqdient}
Let $\alpha \in (0, 1)$ and a $C^1$-diffeomorphism $\phi_0 \in C^{1+\alpha}(\mathbb{S}^1_{2\pi}; \mathbb{S}^1_{2\pi})$ satisfies $\partial_u \phi_0 (u) > 0$ for $u \in \mathbb{S}^1_{2\pi}$. 
Let also $\phi \in C([0, \infty); C^1(\mathbb{S}^1_{2\pi}; \mathbb{S}^1_{2\pi})) \cap C^\infty(\mathbb{S}^1_{2\pi} \times (0, \infty); \mathbb{S}^1_{2\pi})$ be a solution to \eqref{deri-initial-pro}. 
Then, $\phi(\cdot, t): \mathbb{S}^1_{2\pi} \to \mathbb{S}^1_{2\pi}$ is a diffeomorphism and it holds that 
\[ e^{-M(t)} \left(\min_{u \in \mathbb{S}^1_{2\pi}} \partial_u \phi_0 (u) \right) \le \partial_u \phi(u,t) \le e^{M(t)} \left(\max_{u \in \mathbb{S}^1_{2\pi}} \partial_u \phi_0(u)\right) \quad \text{for} \; \; (u, t) \in \mathbb{S}^1_{2\pi} \times [0, \infty), \]
where 
\[ M(t) = \int_0^t \left(\max_{u \in \mathbb{S}^1_{2\pi}} \partial_u F(u, \tau)\right) \; d\tau. \]
\end{prop}

\begin{proof}
Let $v(u,t) = \partial_u \phi(u, t)$. 
Differentiating the parabolic differential equation of \eqref{deri-initial-pro} with respect to $u$, we have 
\begin{equation}\label{eq-gradient}
\begin{aligned} 
\partial_t v(u, t) =&\; \frac{\partial_u^2 v(u, t)}{(v(u, t))^2 (\ell(\phi(u, t), t))^2} - \frac{2 (\partial_u v(u,t))^2}{(v(u, t))^3 (\ell(\phi(u, t), t))^2} \\
&\; - \frac{2 \partial_u \ell(\phi(u, t), t) \partial_u v(u, t)}{v(u,t) (\ell(\phi(u, t), t))^3} - \partial_u F(\phi(u, t), t) v(u, t). 
\end{aligned}
\end{equation}
Therefore, the maximum principle (cf.\ \cite[Theorem 2.1 in Chapter 1]{LSU}) yields the gradient estimate. 

Since the gradient estimate yields $\partial_u \phi(u, t) > 0$ for $(u, t) \in \mathbb{S}^1_{2\pi} \times (0, \infty)$, we can see that $\phi(\cdot, t)$ is a diffeomorphism by showing 
\[ \phi(2\pi, t) - \phi(0, t) = \int_{\mathbb{S}^1_{2\pi}} v(u, t) \; du = 2\pi \quad \text{for} \; \; t \ge 0. \]
Since $\phi_0: \mathbb{S}^1_{2\pi} \to \mathbb{S}^1_{2\pi}$ is a $C^1$-diffeomorphism, we have 
\begin{equation}\label{apriori-gradient1}
\int_{\mathbb{S}^1_{2\pi}} v(u, 0) \; du = 2\pi. 
\end{equation}
On the other hand, \eqref{eq-gradient} yields 
\begin{align*}
\partial_t \int_{\mathbb{S}^1_{2\pi}} v(u, t) \; du = &\; \int_{\mathbb{S}^1_{2\pi}} \frac{\partial_u^2 v(u, t)}{(v(u, t))^2 (\ell(\phi(u, t), t))^2} - \frac{2 (\partial_u v(u,t))^2}{(v(u, t))^3 (\ell(\phi(u, t), t))^2} \\
&\; \qquad - \frac{2 \partial_u \ell(\phi(u, t), t) \partial_u v(u, t)}{v(u,t) (\ell(\phi(u, t), t))^3} - \partial_u F(\phi(u, t), t) v(u, t) \; du \\
=&\; \left[ \frac{\partial_u^2 \phi(u, t)}{(\partial_u \phi(u, t))^2 (\ell(\phi(u, t), t))^2} - F(\phi(u, t), t) \right]_{u=0}^{2\pi} = 0. 
\end{align*}
Therefore, \eqref{apriori-gradient1} implies $\int_{\mathbb{S}^1_{2\pi}} v(u, t) \; du = 2\pi$ for $t \ge 0$. 
\end{proof}

Due to the a priori estimate in Proposition \ref{prop:apriori-grqdient}, we can apply a standard existence theory for quasi-linear parabolic equations to prove Proposition \ref{prop:re-para-time} as follows: 

\begin{proof}[Proof of Proposition \ref{prop:re-para-time}] 
Due to Proposition \ref{prop:re-para}, there exist $n\in \mathbb{N}$ and a $C^1$-diffeomorphism $\phi_0 \in C^{1+\alpha}(\mathbb{S}^1_{2\pi}; \mathbb{S}^1_{2\pi})$ such that $(\tilde{X}_0, \tilde{\nu}_0) := (X_0 \circ \phi_0, \nu_0 \circ \phi_0)$ satisfies 
\[ \tilde{\nu}_0(u) = \begin{pmatrix}
\sin(nu) \\
-\cos(nu)
\end{pmatrix} \quad \text{for} \; \; u \in \mathbb{S}^1_{2\pi}. \]
Then, since Proposition \ref{prop:apriori-grqdient} and the $\ell$-convexity imply uniformly parabolicity of \eqref{deri-initial-pro}, a standard existence theory for quasi-linear parabolic equations (cf.\ \cite[Section 6 in Chapter XII]{Li}) yields that a global-in-time solution $\phi \in C([0, \infty); C^1(\mathbb{S}^1_{2\pi}; \mathbb{S}^1_{2\pi})) \cap C^\infty(\mathbb{S}^1 \times (0, \infty); \mathbb{S}^1_{2\pi})$ to \eqref{deri-initial-pro} uniquely exists. 
Proposition \ref{prop:apriori-grqdient} also implies that $\phi(\cdot, t):\mathbb{S}^1_{2\pi} \to \mathbb{S}^1_{2\pi}$ is a diffeomorphism for $t \ge 0$. 
The calculation as in \eqref{deri-PDE} shows that the re-formulated Legendre curve $(\tilde{X}, \tilde{\nu})$ satisfies the latter equality in \eqref{re-para-time}. 
The former one in \eqref{re-para-time} follows from the choice of $\phi_0$. 
\end{proof}

In order to prove Theorem \ref{thm:initial problem}, due to Proposition \ref{prop:re-para-time}, it is sufficient to study on inverse curvature flows satisfying 
\begin{equation}\label{re-flow-eq} 
\partial_t X = \frac{\beta}{\ell} \nu + \frac{\partial_u \beta}{\ell^2} \mu \quad \text{on} \; \; \mathbb{S}^1_{2\pi} \times (0, \infty) 
\end{equation}
and assuming that the initial Legendre curve $(X_0, \nu_0)$ satisfies 
\begin{equation}\label{initial-as} 
\nu_0(u) = \begin{pmatrix}
\sin(nu) \\
-\cos(nu)
\end{pmatrix} \quad \text{for} \; \; u \in \mathbb{S}^1_{2\pi} 
\end{equation}
for some $n \in \mathbb{N}$. 
Proposition \ref{prop:eq-beta} and \ref{prop:eq-l} also yield that the flow satisfies 
\begin{align} 
& \partial_t \beta = \partial_u \left(\frac{\partial_u \beta}{\ell^2}\right) + \beta \quad \text{on} \; \; \mathbb{S}^1_{2\pi} \times (0, \infty), \label{eq-flow-beta}\\
& \partial_t \ell = \partial_u \left(\frac{\partial_u l}{\ell^2} \right) \quad \text{on} \; \; \{(u,t) \in \mathbb{S}^1_{2\pi} \times (0, \infty): \partial_u \beta(u, t) \neq 0\}. \label{eq-flow-l}
\end{align}
Therefore, in order to prove the unique existence result stated in Theorem \ref{thm:initial problem}, we first solve an initial value problem of a system corresponding to the above equations, and we then construct the flow $(X, \nu)$ from its solution $(\ell, \beta)$. 
We finally prove the construction of $(X, \nu)$ from $(\ell, \beta)$ yields the uniqueness of the flow. 
The properties on the zero points of $\beta(\cdot, t)$ stated in Theorem \ref{thm:initial problem} follows from the zero number diminishing property for parabolic equations (cf.\ \cite{An, Ma}). 

We first present the unique existence theory for the following system: 
\begin{numcases}{}
\partial_t \beta = \partial_u \left(\tfrac{\partial_u \beta}{\ell^2}\right) + \beta & \text{on $\mathbb{S}^1_{2\pi} \times (0, \infty)$}, \label{sys-beta}\\
\partial_t \ell = \partial_u \left(\tfrac{\partial_u \ell}{\ell^2} \right) & \text{on $\mathbb{S}^1_{2\pi} \times (0, \infty)$}, \label{sys-l}\\
\beta(\cdot, 0) = \beta_0, \; \ell(\cdot, 0) \equiv n & \text{on $\mathbb{S}^1_{2\pi}$}. \label{sys-initial}
\end{numcases}
Here, $\beta_0 \in C^\alpha(\mathbb{S}^1_{2\pi})$ and $n \in \mathbb{N}$ correspond to the Legendre curvature of $(X_0, \nu_0) \in C^{1+\alpha}(\mathbb{S}^1_{2\pi}; \mathbb{R}^2 \times \mathbb{S}^1)$. 

\begin{prop}\label{prop:existence-system}
Let $\beta_0 \in C^\alpha(\mathbb{S}^1_{2\pi})$ be a non-zero function and $n \in \mathbb{N}$. 
Then, there exists a unique global-in-time solution $(\ell, \beta) \in C(\mathbb{S}^1_{2\pi} \times [0, \infty); \mathbb{R}^2) \cap C^\infty(\mathbb{S}^1_{2\pi} \times (0, \infty); \mathbb{R}^2)$ to \eqref{sys-beta}--\eqref{sys-initial} whose concrete representation is given by 
\begin{equation} \label{ex-form-l-beta} 
\ell(u, t) = n, \; \; \beta(u,t) = e^{t} a_0 + \sum_{k=1}^\infty e^{(1-\frac{k^2}{n^2})t} \left(a_k \cos(ku) + b_k \sin(ku)\right) 
\end{equation}
for $(u,t) \in \mathbb{S}^1_{2\pi} \times [0, \infty)$, where $\{a_k\}_{k \in \mathbb{N} \cup \{0\}}$ and $\{b_k\}_{k \in \mathbb{N}}$ are respectively the family of the Fourier cosine coefficient and Fourier sine coefficient of $\beta_0$. 

Furthermore, the zero points of $\beta(\cdot, t)$ are isolated for any $t>0$ and the number of the zero points $z(t)$ of $\beta(\cdot, 0)$ is non-increasing with respect to $t > 0$. 
Moreover, if $\beta(u_0, t_0) = 0$ and $\partial_u \beta(u_0, t_0) = 0$ at $(u_0, t_0) \in \mathbb{S}^1_{2\pi} \times (0, \infty)$, then $z(t)$ is strictly decreasing at $t=t_0$ in the sense stated in Theorem \ref{thm:initial problem}. 
\end{prop}

\begin{proof}
By a standard existence theory for quasi-linear parabolic equations, the equation \eqref{sys-l} with $\ell(\cdot, 0) \equiv n$ has a smooth unique solution and it is obviously $\ell \equiv n$ on $\mathbb{S}^1_{2\pi} \times [0, \infty)$. 
Then, \eqref{sys-beta} can be re-written as 
\begin{equation}\label{sys-beta2}
\partial_t \beta = \frac{1}{n^2} \partial_u^2 \beta + \beta \quad \text{on} \; \; \mathbb{S}^1_{2\pi} \times (0, \infty). 
\end{equation}
Therefore, the maximum principle yields a priori estimate 
\[ |\beta(u,t)| \le e^t \sup_{u \in \mathbb{S}^1_{2\pi}} |\beta_0(u)| \quad \text{for} \; \; (u, t) \in \mathbb{S}^1_{2\pi} \times [0, \infty), \]
and thus the linear equation \eqref{sys-beta2} with the initial condition $\beta(\cdot, 0) = \beta_0$ has a unique solution $\beta \in C(\mathbb{S}^1_{2\pi} \times [0, \infty)) \cap C^\infty(\mathbb{S}^1_{2\pi} \times (0, \infty))$ due to a standard existence theory for linear parabolic equations (cf. \cite[Chapter 5]{Lu}). 
Furthermore, Fourier's method can be applied to obtain the representation of $\beta$ in \eqref{ex-form-l-beta}. 
Indeed, letting $\mathcal{L}: \mathcal{D}(\mathcal{L}) \subset L^2(\mathbb{S}^1_{2\pi}) \to L^2(\mathbb{S}^1_{2\pi})$ be 
\begin{equation}\label{def-operator} 
\mathcal{L} f := \frac{1}{n^2} \partial_u^2 f + f, \quad \mathcal{D}(\mathcal{L}) := H^2(\mathbb{S}^1_{2\pi}), 
\end{equation}
where $H^2$ is the Sobolev space $W^{2, 2}$, the eigenvalues $\{\lambda_k\}_{k \in \mathbb{N}\cup\{0\}}$ of $\mathcal{L}$ and the corresponding eigenfunctions $f_0$ and $\{f_{k, i}\}_{k \in \mathbb{N}, i=1,2}$ are given by 
\begin{equation}\label{eigen}
\lambda_k = 1 - \frac{k^2}{n^2}, \quad f_0(u) = 1, \quad f_{k,1}(u) = \cos (ku), \quad f_{k, 2}(u) = \sin (ku). 
\end{equation}
Thus, a standard Fourier's method (cf.\ \cite[Section 10.1]{St} and \cite[Section 22]{We}) yields the representation of $\beta$ in \eqref{ex-form-l-beta}. 
Therefore, the unique existence and the representation of $(\ell, \beta)$. 
The remained claim on the zero points of $\beta$ can be obtained by applying the zero number diminishing theory as in \cite{An, Ma} to \eqref{sys-beta2}. 
\end{proof}

We next uniquely construct a flow satisfying \eqref{re-flow-eq} from the solution $(\ell, \beta)$ obtained in Proposition \ref{prop:existence-system}. 

\begin{prop}\label{prop:special-flow}
Let $\alpha \in (0, 1)$ and $(X_0, \nu_0) \in C^{1+\alpha}(\mathbb{S}^1_{2\pi}; \mathbb{R}^2 \times \mathbb{S}^1)$ be a Legendre curve satisfying \eqref{initial-as} for some $n \in \mathbb{N}$, such that $X_0$ does not describe a point. 
Then, there exists a unique flow $(X, \nu) \in C([0, \infty); C^1(\mathbb{S}^1_{2\pi}; \mathbb{R}^2 \times \mathbb{S}^1)) \cap C^\infty(\mathbb{S}^1_{2\pi} \times (0, \infty); \mathbb{R}^2 \times \mathbb{S}^1)$ satisfying \eqref{re-flow-eq} and starting from $(X_0, \nu_0)$. 
\end{prop}

\begin{proof}
We first construct $(X, \nu)$. 
The uniqueness will be proved later. 

Let $(\ell_0, \beta_0)$ be the Legendre curvature of $(X_0, \nu_0)$. 
Then, $\ell_0 \equiv n$ on $\mathbb{S}^1_{2\pi}$ and $\beta_0 \in C^{1+\alpha}(\mathbb{S}^1_{2\pi})$. 
Furthermore, since $X_0$ does not describe a point, $\beta_0$ is non-zero function. 
Therefore, due to Proposition \ref{prop:existence-system}, we can obtain the unique solution $(\ell, \beta)$ to \eqref{sys-beta}--\eqref{sys-initial}. 
We define $(X, \nu)$ by 
\begin{align}
&\nu(u, t) := \begin{pmatrix}
\sin(nu) \\
-\cos(nu)
\end{pmatrix} \quad \text{for} \; \; (u, t) \in \mathbb{S}^1_{2\pi} \times [0, \infty), \label{construct-nu}\\
&X(u, t) := \int_0^t \frac{\beta(u, \tau)}{\ell(u, \tau)} \nu(u, \tau) + \frac{\partial_u \beta(u, \tau)}{(\ell(u, \tau))^2} \mu(u, \tau) \; d\tau  + X(u, 0) \quad \text{for} \; \; (u, t) \in \mathbb{S}^1_{2\pi}, \label{construct-X}
\end{align}
where 
\[ \mu(u,t) := J\nu(u, t) = \begin{pmatrix}
\cos(nu) \\
\sin(nu)
\end{pmatrix} \quad \text{for} \; \; (u, t) \in \mathbb{S}^1_{2\pi} \times [0, \infty). \]
Then, we can easily see that 
\[ \partial_u \nu = n \mu = \ell \mu \quad \text{on} \; \; \mathbb{S}^1_{2\pi} \times [0, \infty). \]
Furthermore, due to $\ell \equiv n$ and \eqref{sys-beta}, we have 
\begin{align*}
\partial X(u, t) =&\; \int_0^t \partial_u \left(\frac{\beta(u, \tau)}{\ell(u, \tau)} \nu(u, \tau) + \frac{\partial_u \beta(u, \tau)}{(\ell(u, \tau))^2} \mu(u, \tau)\right) \; d\tau + \partial_u X(u, 0)\\
=&\; \int_0^t \left(\beta(u, \tau) + \frac{\partial_u^2 \beta(u, \tau)}{(\ell(u, \tau))^2} \right) \mu(u, \tau)  \; d\tau + \beta(u, 0) \mu(u, 0) \\
=&\; \int_0^t (\partial_t\beta(u, \tau)) \mu(u, \tau) \; d\tau + \beta(u, 0) \mu(u, 0) = \beta(u, t) \mu(u, t) 
\end{align*}
for $(u, t) \in \mathbb{S}^1_{2\pi} \times (0, \infty)$. 
Therefore, $(X, \nu)$ is a family of Legendre curves and $(\ell, \beta)$ is the Legendre curvature of $(X, \nu)$. 
Moreover, due to the construction of $X$ as in \eqref{construct-X}, $(X, \nu)$ obviously satisfies \eqref{re-flow-eq}, which yields the existence of the flow. 

We next assume that $(\tilde{X}, \tilde{\nu})$ is another flow satisfying \eqref{re-flow-eq} and starting from $(X_0, \nu_0)$ to prove $(\tilde{X}, \tilde{\nu}) \equiv (X, \nu)$ on $\mathbb{S}^1_{2\pi} \times [0, \infty)$. 
Then, the Legendre curvature $(\tilde{\ell}, \tilde{\beta})$ of $(\tilde{X}, \tilde{\nu})$ satisfies \eqref{eq-flow-beta} and \eqref{eq-flow-l}. 
Therefore, the zero number diminishing theory can be applied to \eqref{eq-flow-beta} to see that the zero points of $\tilde{\beta}(\cdot, t)$ are isolated. 
Thus, due to the positivity of $\tilde{\ell}$ and the continuity of $\tilde{\ell}$ and its derivatives, we can see that $\tilde{\ell}$ satisfies \eqref{eq-flow-l} on $\mathbb{S}^1_{2\pi} \times (0, \infty)$. 
Then, the uniqueness result in Proposition \ref{prop:existence-system} yields that $(\tilde{\ell}, \tilde{\beta}) \equiv (\ell, \beta)$. 
Since $\ell \equiv n$ and the zero points of $\beta(\cdot, t)$ are isolated, \eqref{eq-nu} yields $\partial_t \tilde{\nu} = 0$ on $\mathbb{S}^1_{2\pi} \times (0, \infty)$. 
Therefore, $\tilde{\nu}$ is given by \eqref{construct-nu}. 
Since $(\tilde{X}, \tilde{\nu})$ satisfies \eqref{re-flow-eq}, we can also see that $\tilde{X}$ is given by \eqref{construct-X}. 
We thus obtain $(\tilde{X}, \tilde{\nu}) \equiv (X, \nu)$ on $\mathbb{S}^1_{2\pi} \times [0, \infty)$, which complete the proof of the uniqueness of the flow. 
\end{proof}

Theorem \ref{thm:initial problem} can be proved by combining Propositions \ref{prop:tangent-velocity}, \ref{prop:re-para-time}, \ref{prop:existence-system} and \ref{prop:special-flow}. 

\section{Self similar solutions}\label{sec:self similar}

In this section, we study on the self similar solutions as stated in Theorem \ref{thm:self similar}. 
Since the Legendre curvature of self similar solutions satisfying \eqref{n-rotation} is formed by $\ell(u,t) \equiv n$ and $\beta(u,t) = \lambda^*(t) \beta^*(u)$, by substituting it into the differential equation \eqref{sys-beta} of $\beta$, we can obtain an exact representation of $\lambda^*$ and $\beta^*$ due to the spectrum of $\mathcal{L}$ as in \eqref{eigen}. 
The self similar solutions can be obtained in line with the construction method of the flow of Legendre curves from the Legendre curvature as in the proof of Theorem \ref{thm:initial problem}. 
The rigorous proof is given below.

\begin{proof}[Proof of Theorem \ref{thm:self similar}]
Let $(\lambda^*(t) X^*(u), \nu^*(u))$ be a self similar solution satisfying \eqref{n-rotation}. 
We prove that $\lambda^*$ and $X^*$ satisfy \eqref{self-similar-lambda} and \eqref{self-similar-X}, respectively. 
Since the Legendre curvature $(\ell, \beta)$ of $(\lambda^*(t) X^*(u), \nu^*(u))$ is given by 
\[ \ell(u, t) = n, \; \; \beta(u, t) = \lambda^*(t) \beta^*(u) \quad \text{for} \; \; (u, t) \in \mathbb{S}^1_{2\pi}. \]
Therefore, by substituting $\ell \equiv n$ and $\beta = \lambda^* \beta^*$ into \eqref{eq-flow-beta}, we have the following differential equation of $\lambda^*$ and $\beta^*$: 
\begin{equation}\label{eq-self-similar}
(\partial_t \lambda^*(t)) \beta^*(u) = \lambda^*(t) \left(\frac{1}{n^2} \partial_u^2 \beta^*(u) + \beta^*(u)\right) \quad \text{for} \; \; (u, t) \in \mathbb{S}^1_{2\pi} \times (0, \infty). 
\end{equation}
Since the eigenvalues and eigenfunctions of the operator $\mathcal{L}$ defined by \eqref{def-operator} are given by \eqref{eigen}, the solutions to \eqref{eq-self-similar} are given by 
\begin{equation}\label{self-similar1} 
\begin{aligned}
&\lambda^*_{n, m, \lambda_0} (t) = \lambda_0 e^{(1-\frac{m^2}{n^2})t}, \quad \beta^*_{0, C_1, 0} (u) = C_1, \\
&\beta^*_{m, C_1, C_2}(u) = C_1\cos(mu) + C_2 \sin(mu) 
\end{aligned}
\end{equation}
for $m \in \mathbb{N} \cup \{0\}$ and $\lambda_0, C_1, C_2 \in \mathbb{R}$. 
Since our purpose is to construct $X(u,t) = \lambda^*(t) X^*(u)$, $\lambda_0$ hereafter can be normalized so that $\lambda_0 = 1$ without loss of generality, and thus $\lambda^*$ satisfies \eqref{self-similar-lambda}. 
We also hereafter assume $(C_1, C_2) \in \mathbb{R}^2 \setminus\{(0, 0)\}$ when $m \in \mathbb{N}$ and $C_1 \neq 0$ when $m = 0$ for $X^*$ to be not a point. 

We first prove that $\beta^*$ does not satisfy \eqref{self-similar1} with $m = n$. 
Assume $\beta^*$ satisfies \eqref{self-similar1} with $m=n, \lambda_0 = 1$ and $(C_1, C_2) \in \mathbb{R}^2 \setminus\{(0, 0)\}$ for a contradiction. 
Let 
\[ \mu^*(u) := J\nu^*(u) = \begin{pmatrix}
\cos(nu) \\
\sin(nu)
\end{pmatrix} \quad \text{for} \; \; u \in \mathbb{S}^1_{2\pi}. \]
Since $X^*$ is a closed curve, we have 
\begin{equation}\label{self-similar-closed} 
0 = X^*(2\pi) - X^*(0) = \int_{\mathbb{S}^1_{2\pi}} \beta^*(u) \mu^*(u) \; du. 
\end{equation}
On the other hand, substituting the formula \eqref{self-similar1} of $\beta^*$ with $n=m$ into the right hand side, we have 
\[ \int_{\mathbb{S}^1_{2\pi}} \beta^*(u) \mu^*(u) \; du = \int_{\mathbb{S}^1_{2\pi}} (C_1\cos(nu) + C_1\sin(nu)) \begin{pmatrix}
\cos(nu) \\
\sin(nu)
\end{pmatrix} \; du \neq 0, \]
which contradicts to \eqref{self-similar-closed}. 
Therefore, the claim holds. 

We next prove that $X^*$ satisfies \eqref{self-similar-X} when $\lambda^*$ and $\beta^*$ are given by \eqref{self-similar1} with $m \in \mathbb{N}\setminus\{n\}, \lambda_0 = 1$ and $(C_1, C_2) \in \mathbb{R}^2 \setminus\{(0, 0)\}$. 
Since $(\lambda^*(t) X^*(u), \nu^*(u))$ is an inverse curvature flow, the construction \eqref{construct-X} of the inverse curvature flow from its Legendre curvature as in the proof of Proposition \ref{prop:special-flow}, we can see that $X^*$ satisfies
\begin{align*}
\lambda^*(t) X^*(u) =&\; \int_0^t \frac{\lambda^*(\tau) \beta^*(u)}{n} \nu^*(u) + \frac{\lambda^*(\tau) \partial_u \beta^*(u)}{n^2} \mu^*(u) \; d\tau + X^*(u) \\
=&\; (\lambda^*(t) - 1)\Bigg(\frac{1}{n(1-\frac{m^2}{n^2})} (C_1\cos(mu) + C_2\sin(mu)) \begin{pmatrix} 
\sin(nu)\\
-\cos(nu)
\end{pmatrix} \\
&\; \qquad + \frac{m}{n^2(1-\frac{m^2}{n^2})} (-C_1 \sin(mu) + C_2 \cos(mu)) \begin{pmatrix}
\cos(nu)\\
\sin(nu)
\end{pmatrix}\Bigg) + X^*(u). 
\end{align*}
Therefore, we have 
\begin{align*}
X^*(u) = &\; \frac{n}{n^2-m^2} (C_1\cos(mu) + C_2\sin(mu)) \begin{pmatrix} 
\sin(nu)\\
-\cos(nu)
\end{pmatrix} \\
&\; + \frac{m}{n^2-m^2} (-C_1 \sin(mu) + C_2 \cos(mu)) \begin{pmatrix}
\cos(nu)\\
\sin(nu)
\end{pmatrix},
\end{align*}
which yields \eqref{construct-X}. 
When $\lambda^*$ and $\beta^*$ satisfies \eqref{self-similar1} with $m=0$, a similar argument yields \eqref{construct-X} with $m=0$ and $C_2=0$. 

Due to the construction of $\lambda^*$ and $X^*$ as above, we obtain the equivalence stated in Theorem \ref{thm:self similar}. 
\end{proof}

\begin{rmk}\label{rmk:self similar}
Due to the form of $\beta^*$ in \eqref{self-similar1}, we can see that $\beta^*$ has $2m$ zero points if $m \in \mathbb{N}$. 
On the other hand, by using the trigonometric addition formula, $X^*$ can be re-written as 
\begin{align*} 
X^* (u) =&\; \frac{C_1}{2} \begin{pmatrix}
\frac{1}{n+m} \sin(mu + nu) + \frac{1}{n-m} \sin(nu - mu) \\
-\frac{1}{n+m} \cos(nu + mu) - \frac{1}{n-m} \cos(nu - mu)
\end{pmatrix} \\
&\; + \frac{C_2}{2} \begin{pmatrix}
- \frac{1}{n+m} \cos(nu + mu) + \frac{1}{n-m} \cos(nu - mu) \\
- \frac{1}{n+m} \sin(nu + mu) + \frac{1}{n-m} \sin(nu - mu). 
\end{pmatrix}
\end{align*}
Therefore, the greatest common divisor of $n+m$ and $|n-m|$, which is denoted by ${\rm gcd}(n+m, |n-m|)$, represents the number of laps of $X^*$. 
In particular, if both $n$ and $m$ are either even or odd, then the number of laps of $X^*$ is at least $2$. 
These facts show that the number of $(2,3)$-cusps of $X^*$ is given by $2m/{\rm gcd}(n+m, |n-m|)$, and thus the constants $C_1$ and $C_2$ do not affect to the number of the singular cusps of ``image'' of $X^*$. 
See also Figures \ref{figure1}--\ref{figure4} below for examples. 
\end{rmk}

\begin{figure}[h]
\begin{minipage}{0.32 \hsize}
\centering
\includegraphics[width=4.5cm]{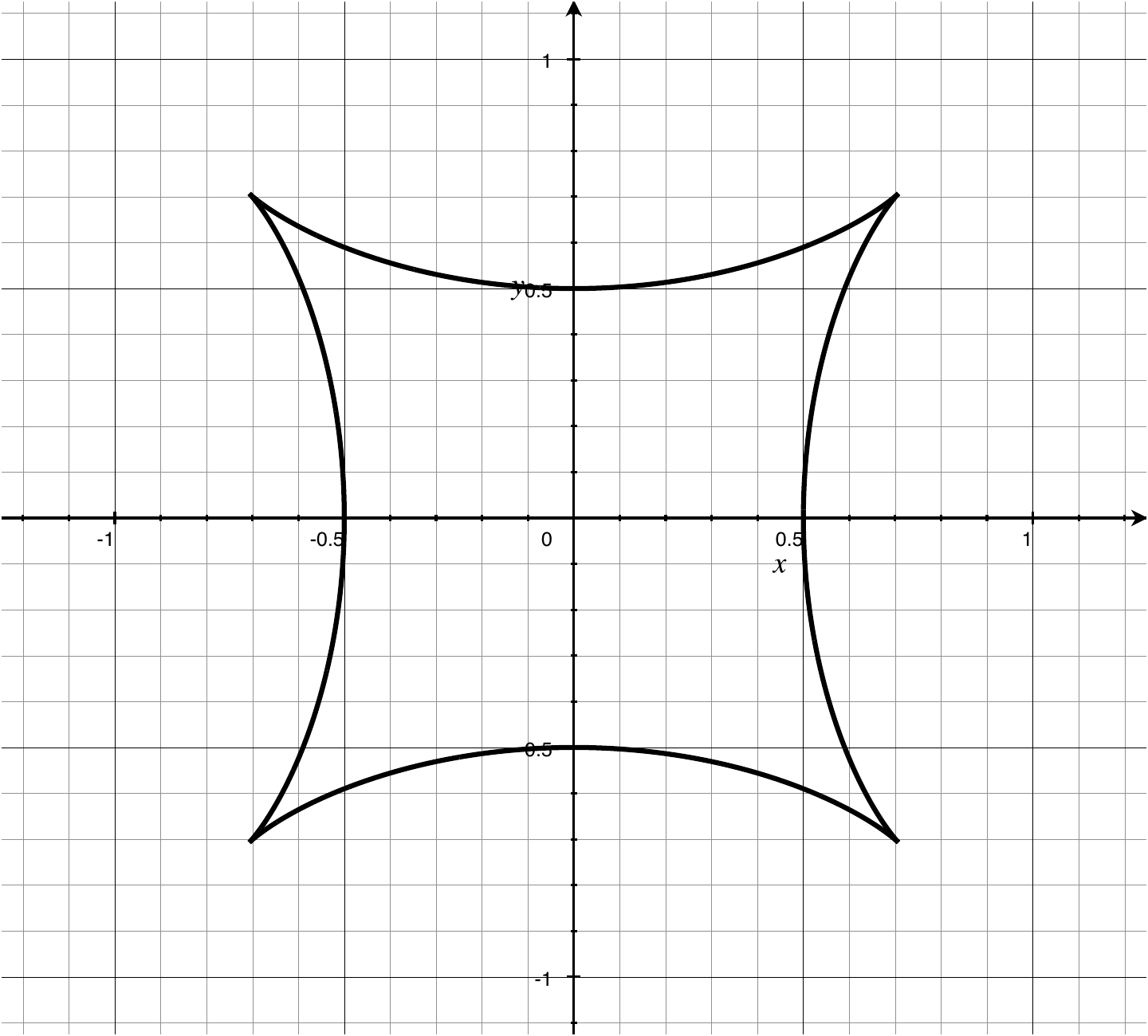}
$m=2, C_1=1.5$
\end{minipage}
\begin{minipage}{0.32 \hsize}
\centering
\includegraphics[width=4.5cm]{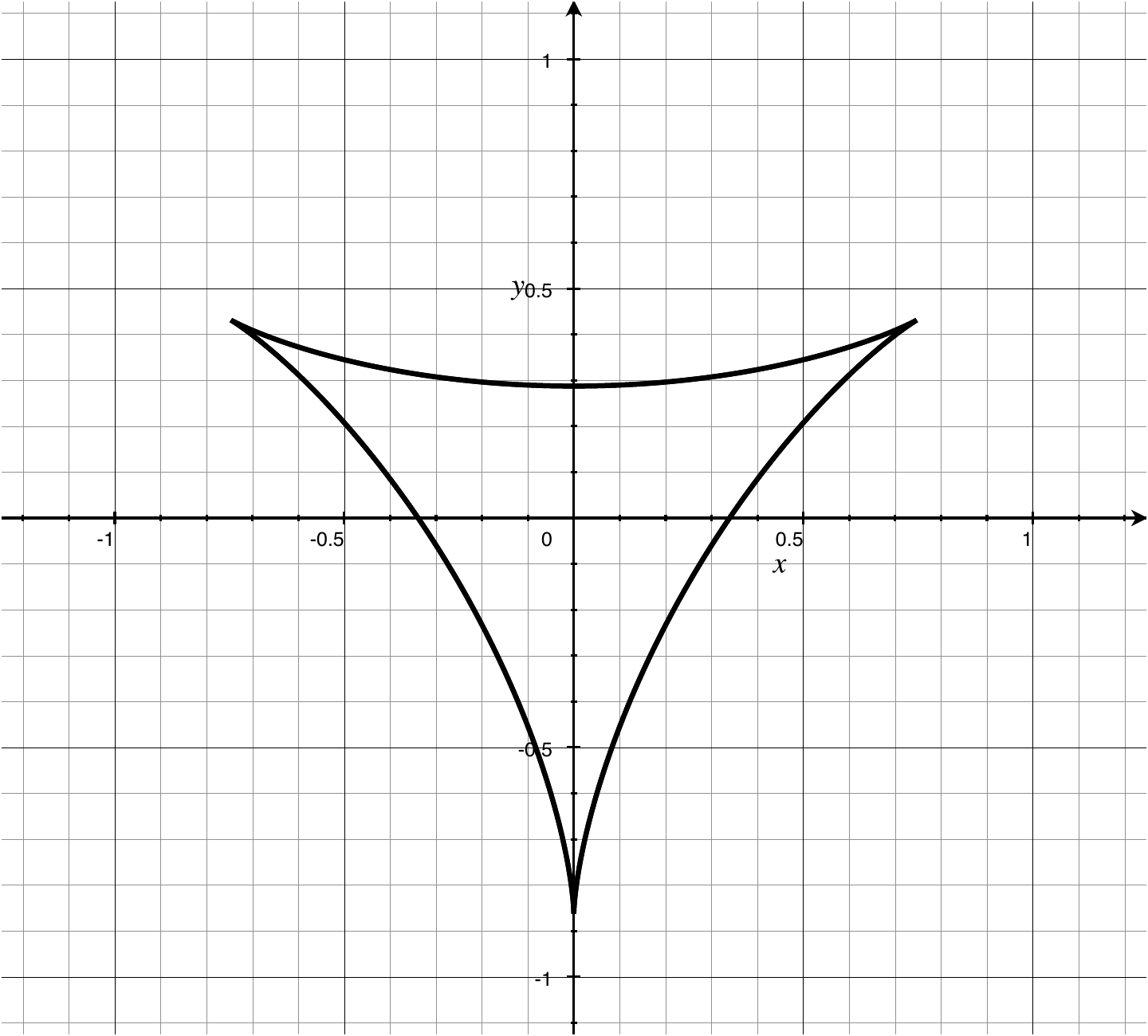}
$m=3, C_1=2.3$
\end{minipage}
\begin{minipage}{0.32 \hsize}
\centering
\includegraphics[width=4.5cm]{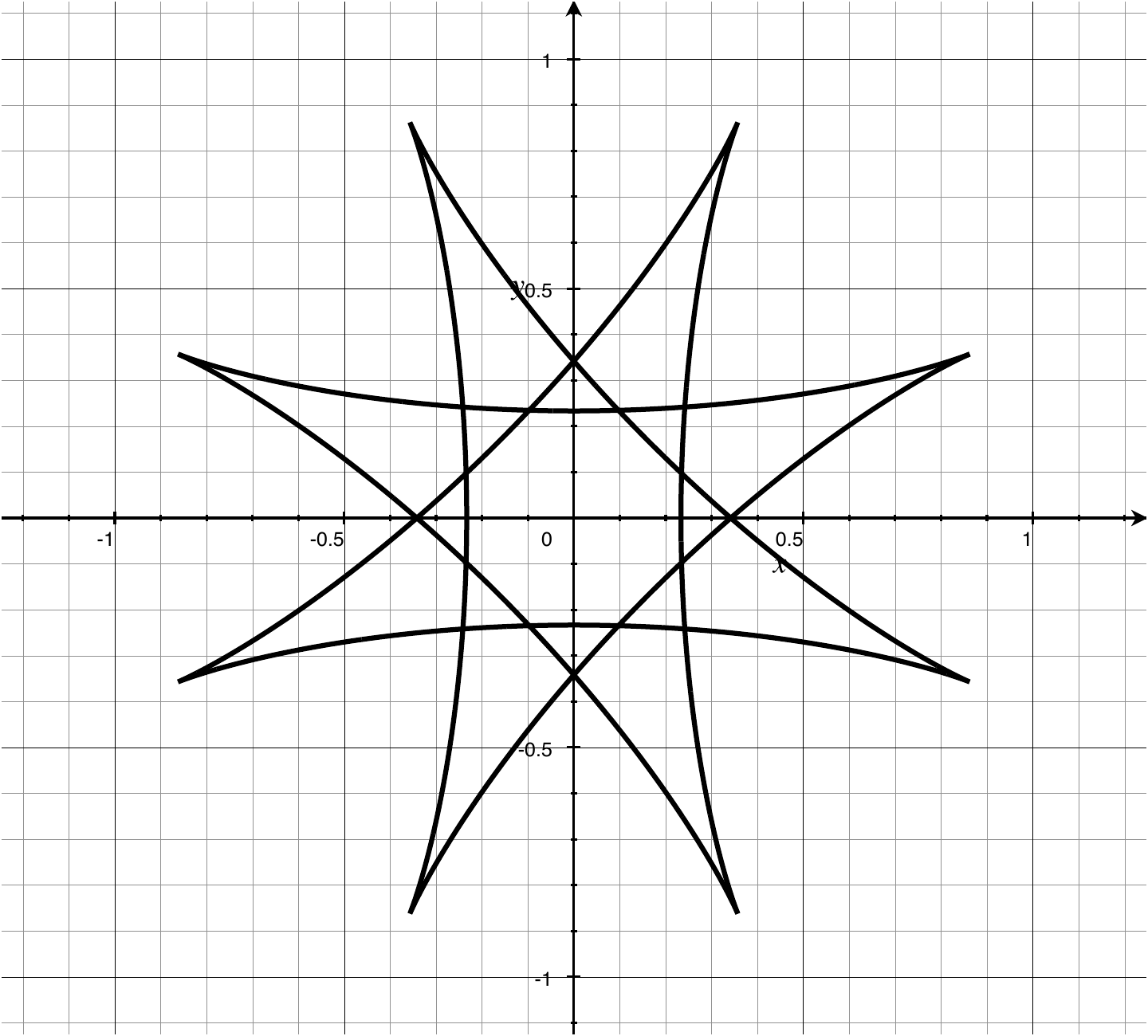}
$m=4, C_1=3.5$
\end{minipage}
\caption{All figures show the image of $X^*$ when $n=1$ and $C_2=0$. The values of the remained parameters are listed below each figure. Only the second one has a lap count of 2.} 
\label{figure1}
\end{figure}

\begin{figure}[h]
\begin{minipage}{0.32 \hsize}
\centering
\includegraphics[width=4.5cm]{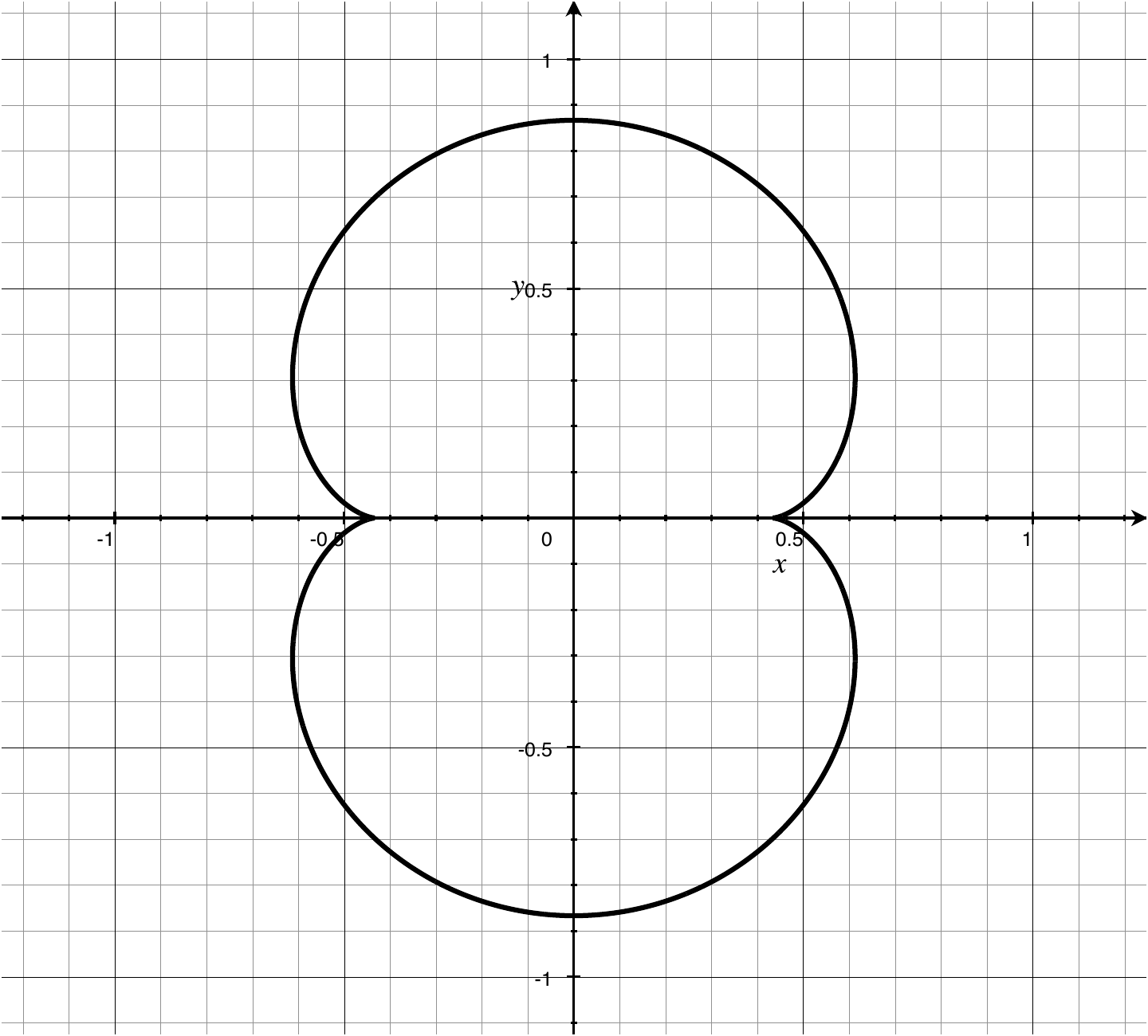}
$m=1, C_1=1.3$
\end{minipage}
\begin{minipage}{0.32 \hsize}
\centering
\includegraphics[width=4.5cm]{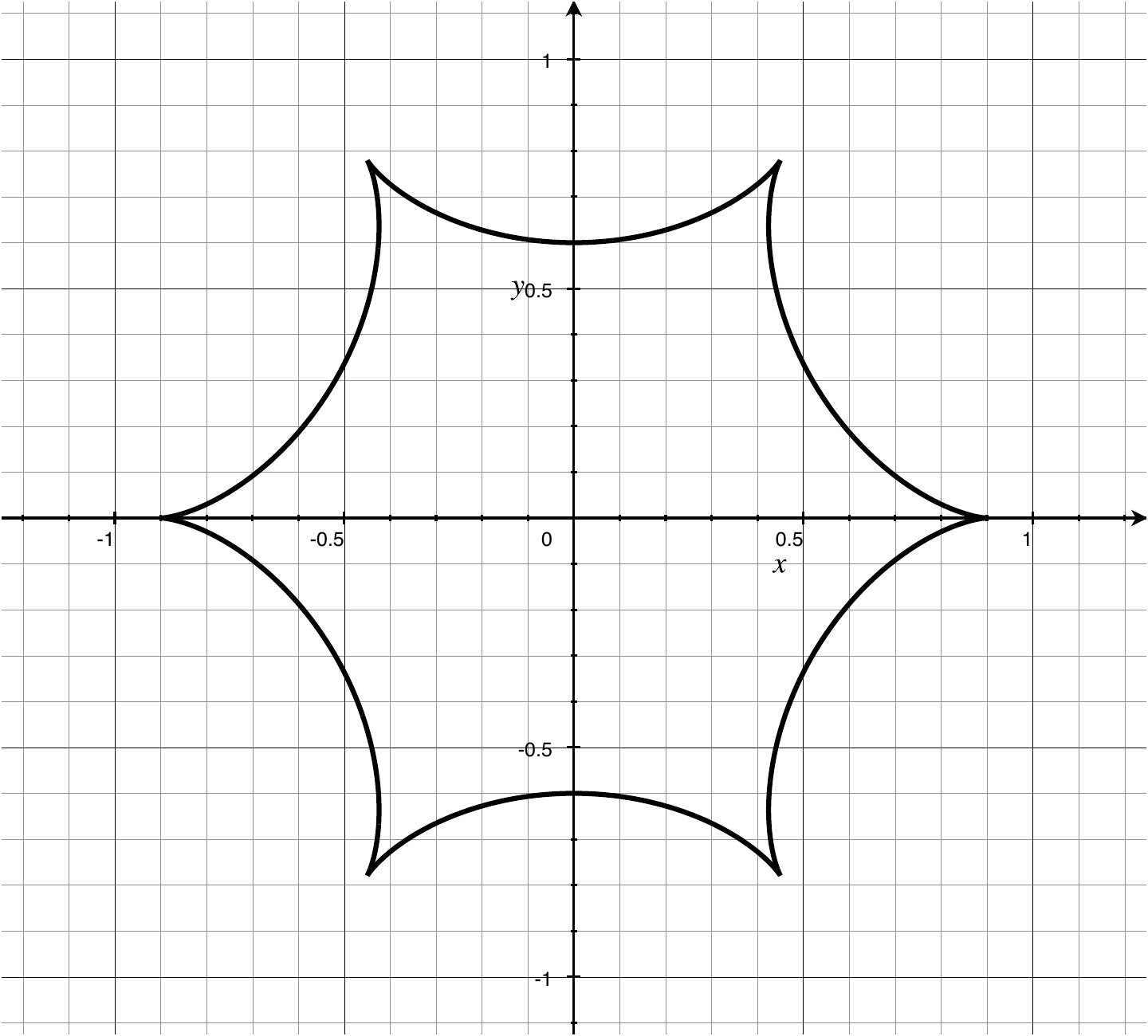}
$m=3, C_1=1.5$
\end{minipage}
\begin{minipage}{0.32 \hsize}
\centering
\includegraphics[width=4.5cm]{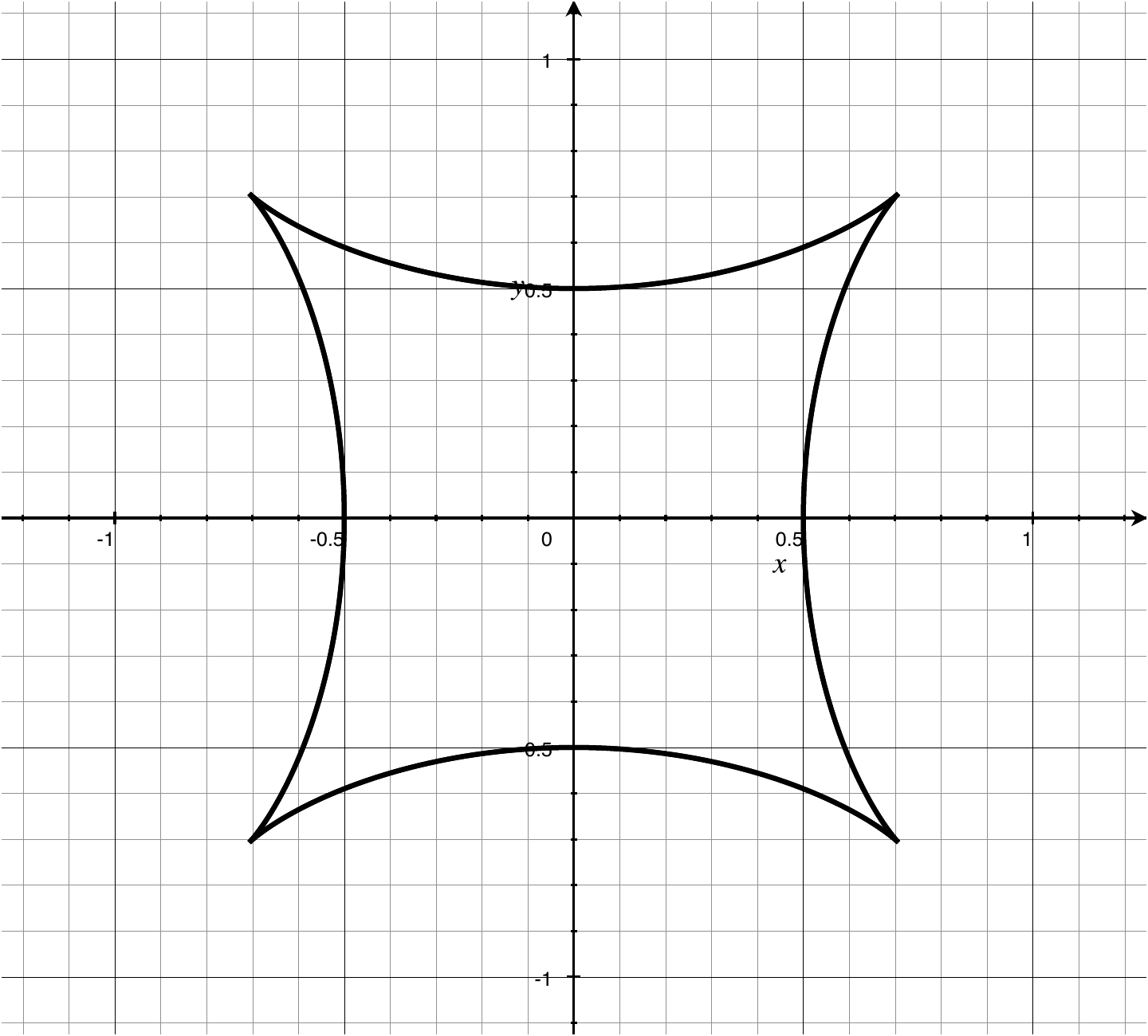}
$m=4, C_1=3$
\end{minipage}
\caption{All figures show the image of $X^*$ when $n=2$ and $C_2=0$. The values of the remained parameters are listed below each figure. Only the third one has a lap count of 2 and the image of $X^*$ coincides with the first one in Figure \ref{figure1}.}
\label{figure2}
\end{figure}

\begin{figure}[h]
\begin{minipage}{0.32 \hsize}
\centering
\includegraphics[width=4.5cm]{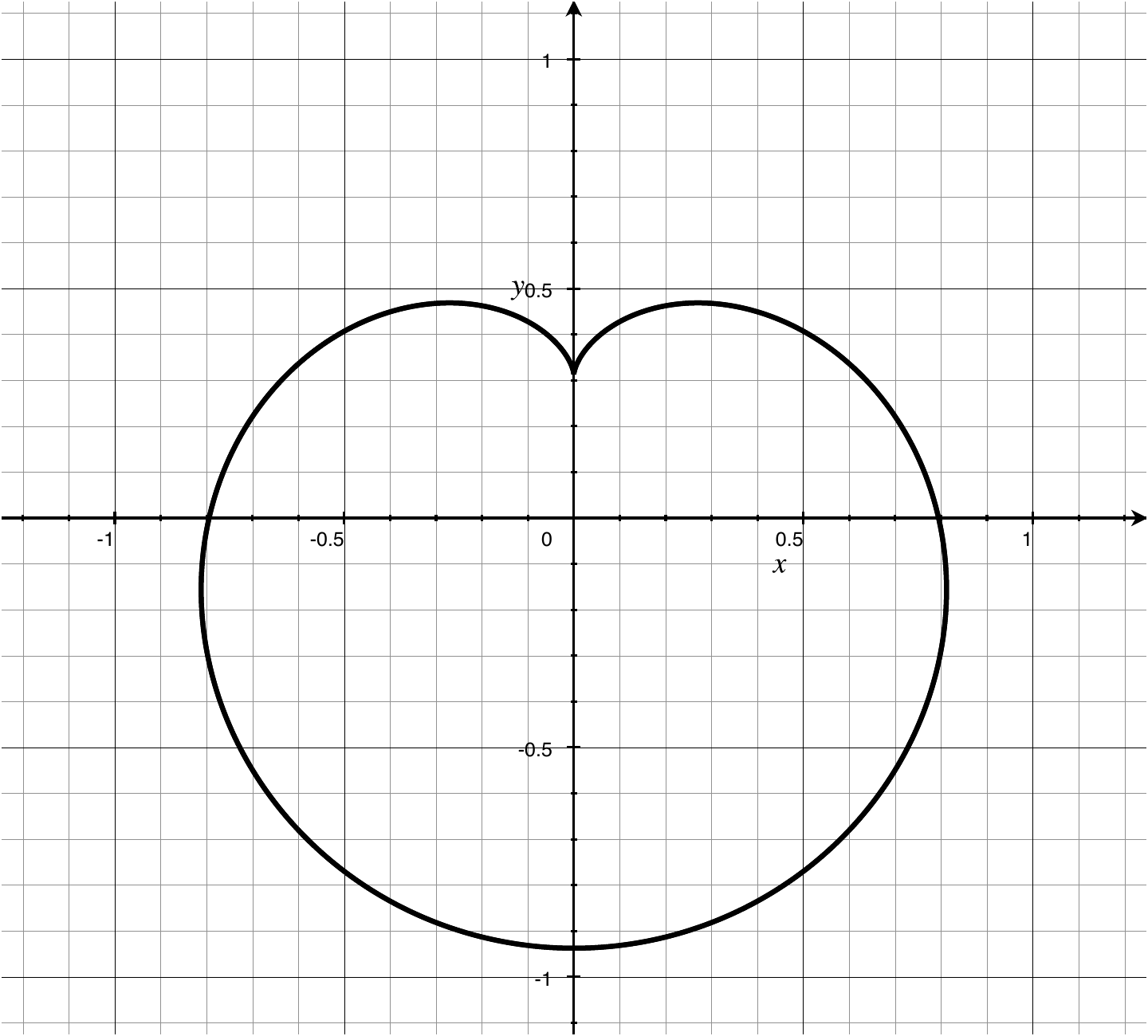}
$m=1, C_1=2.5$
\end{minipage}
\begin{minipage}{0.32 \hsize}
\centering
\includegraphics[width=4.5cm]{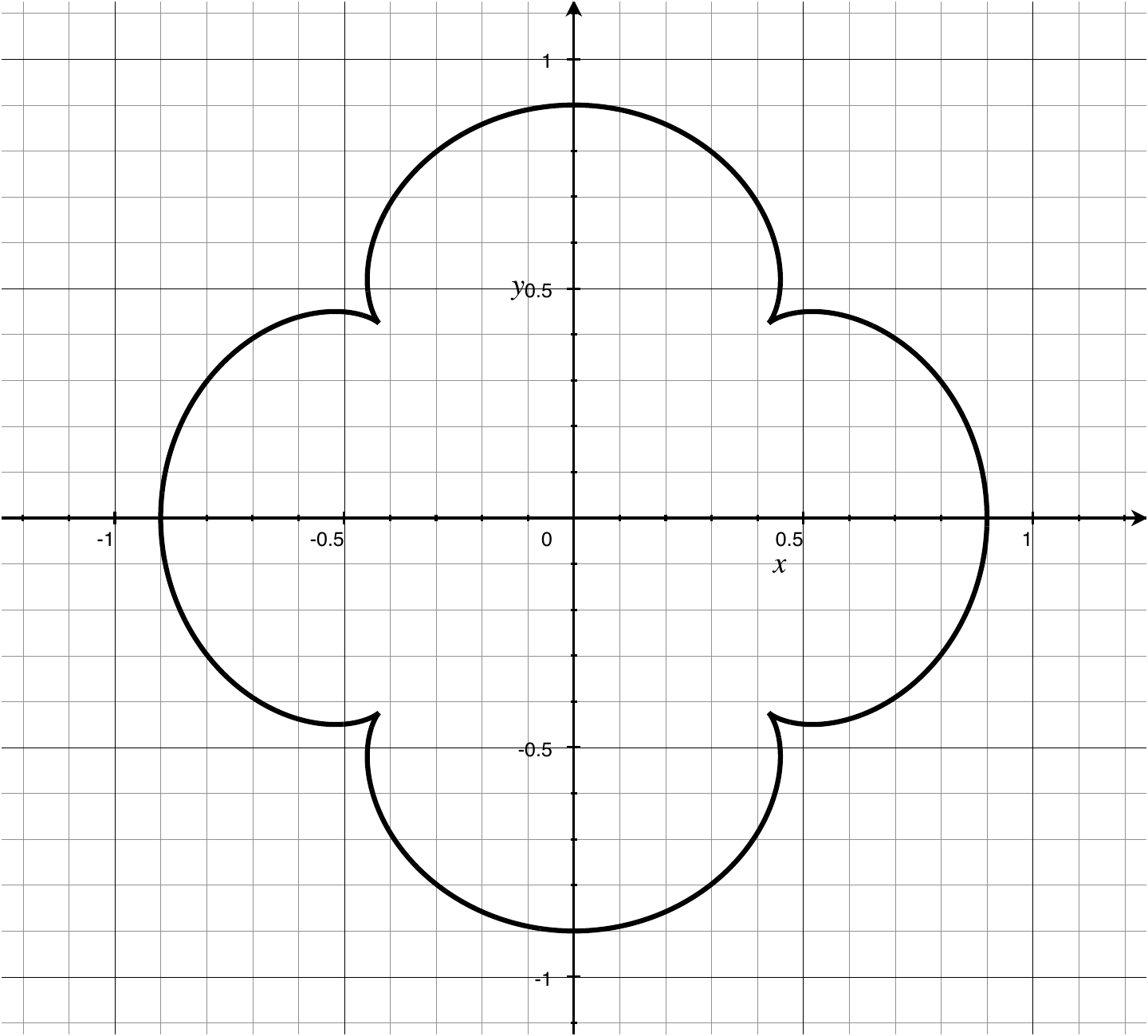}
$m=2, C_1=1.5$
\end{minipage}
\begin{minipage}{0.32 \hsize}
\centering
\includegraphics[width=4.5cm]{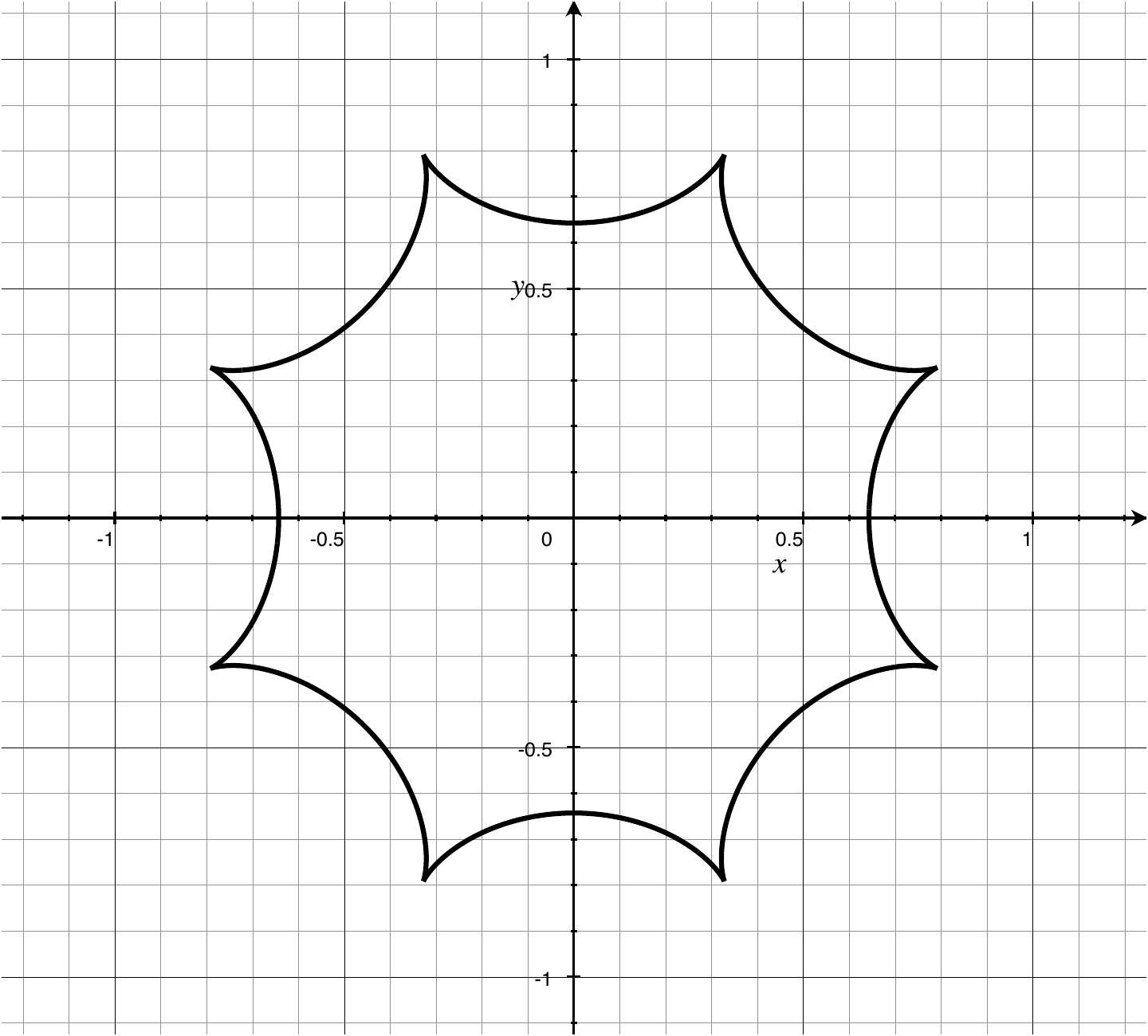}
$m=4, C_1=5$
\end{minipage}
\caption{All figures show the image of $X^*$ when $n=3$ and $C_2=0$. The values of the remained parameters are listed below each figure. Only the first one has a lap count of 2.}
\label{figure3}
\end{figure}

\begin{figure}[h]
\begin{minipage}{0.32 \hsize}
\centering
\includegraphics[width=4.5cm]{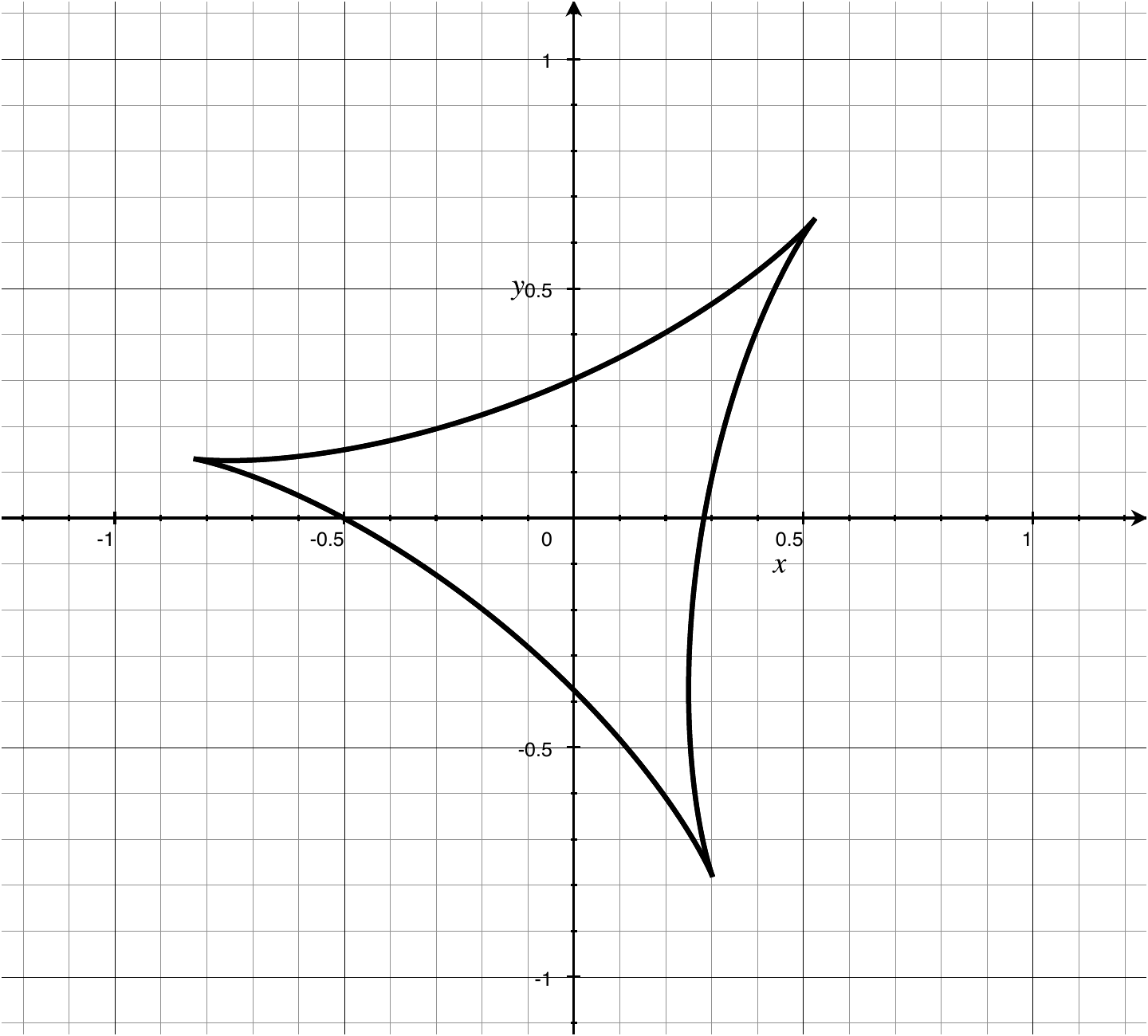}
$C_1=1, C_2=2$
\end{minipage}
\begin{minipage}{0.32 \hsize}
\centering
\includegraphics[width=4.5cm]{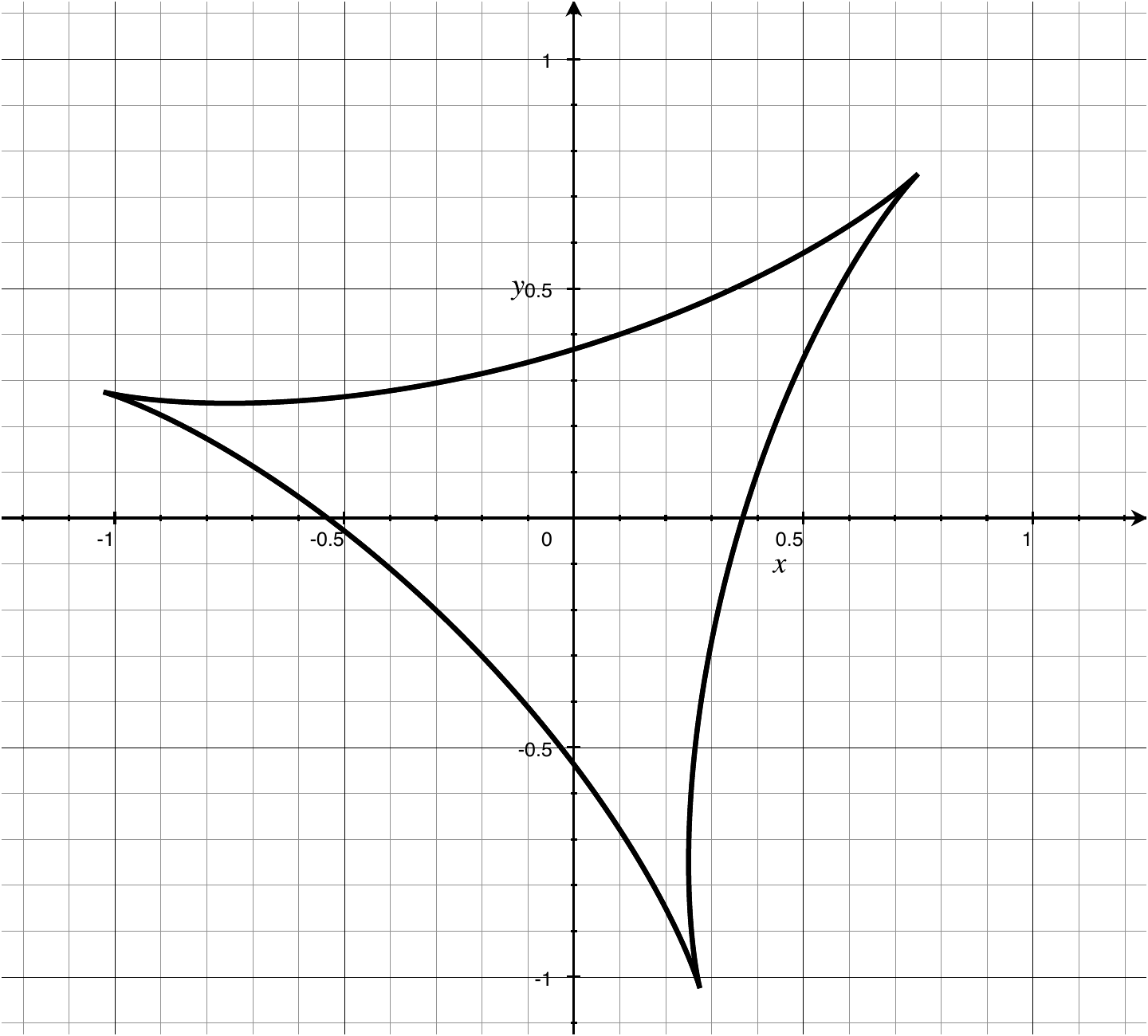}
$C_1=2, C_2=2$
\end{minipage}
\begin{minipage}{0.32 \hsize}
\centering
\includegraphics[width=4.5cm]{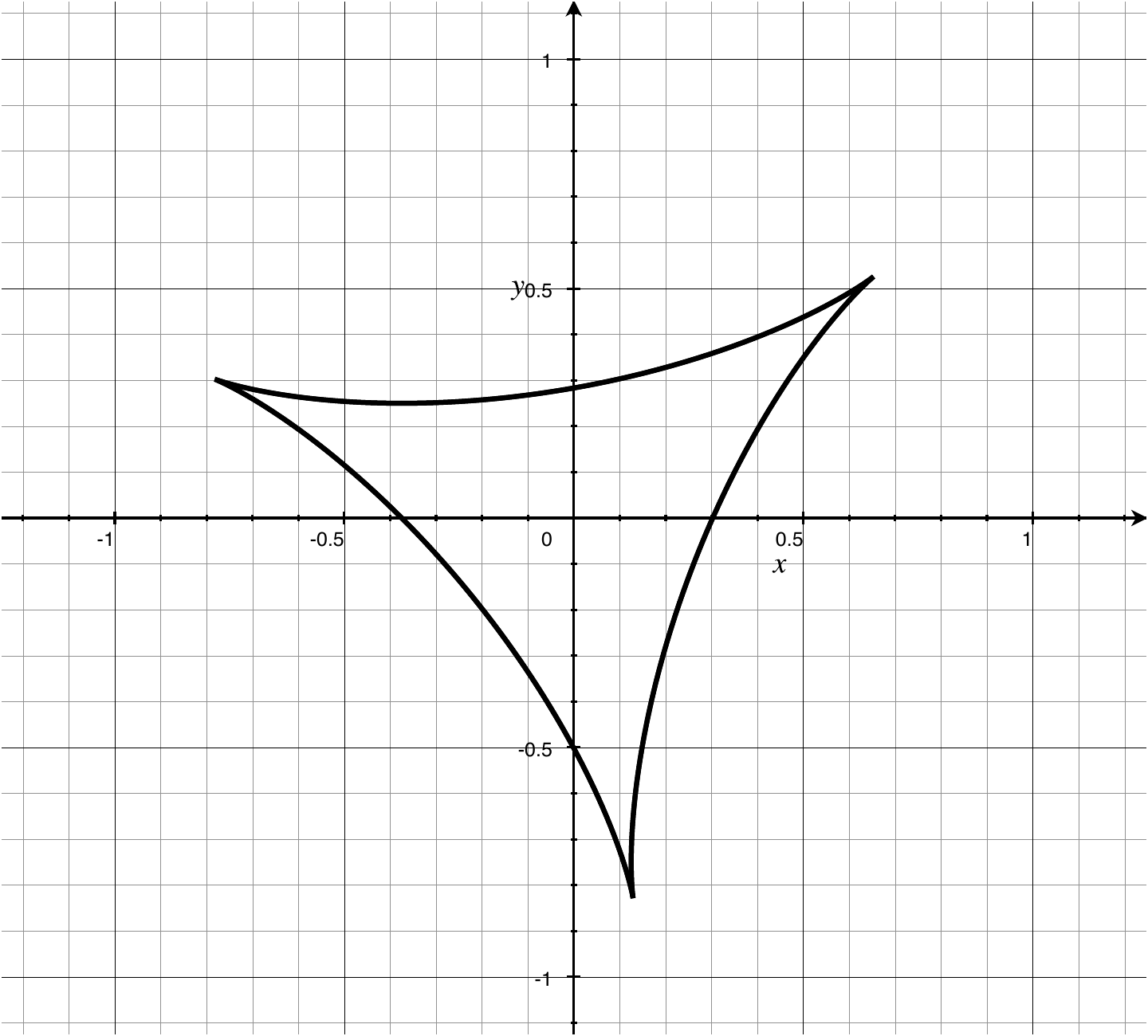}
$C_1=2, C_2 = 1$
\end{minipage}
\caption{All figures show the image of $X^*$ when $n=1$ and $m=3$. Therefore, the number of laps for all curves is 2. To confirm that $C_1$ and $C_2$ do not affect ``visual'' singular cusps, the images output the curves with the values of the remained parameters listed below each figure.}
\label{figure4}
\end{figure}

\section{Asymptotic behavior}\label{sec:asymptotics}

In this section, we analyze the asymptotic behavior of the inverse curvature flow. 
We first present a proposition to prove the statement (i) in Theorem \ref{thm:asymptotics}. 

\begin{prop}\label{prop:fourier-coefficient}
Let $\alpha \in (0,1)$ and $(X_0, \nu_0) \in C^{1+\alpha} (\mathbb{S}^1_{2\pi}; \mathbb{R}^2 \times \mathbb{S}^1)$ be an $\ell$-convex Legendre curve satisfying \eqref{n-rotation} for some $n \in \mathbb{N}$ and $(\ell_0, \beta_0)$ be the Legendre curvature of $(X_0, \nu_0)$. 
Let also $\{a_k\}_{k \in \mathbb{N} \cup \{0\}}$ and $\{b_k\}_{k \in \mathbb{N}}$ be respectively the family of Fourier cosine coefficients and Fourier sine coefficients of $\beta_0$. 
Then, $a_n = b_n = 0$. 
\end{prop}

\begin{proof}
Let 
\[ \mu_0(u) := J\nu_0(u) = \begin{pmatrix}
\cos(nu) \\
\sin(nu)
\end{pmatrix} \quad \text{for} \; \; u \in \mathbb{S}^1_{2\pi}. \]
Since $X_0$ is a closed curve, we have 
\[ 0 = X_0(2\pi) - X_0(0) = \int_{\mathbb{S}^1_{2\pi}} \beta_0(u) \mu_0(u) \; du = \int_{\mathbb{S}^1_{2\pi}} \beta_0(u)  \begin{pmatrix}
\cos(nu) \\
\sin(nu)
\end{pmatrix} \; du. \]
Therefore, the Fourier coefficients 
\[ a_n = \int_{\mathbb{S}^1_{2\pi}} \beta_0(u) \cos(nu) \; du, \quad b_n = \int_{\mathbb{S}^1_{2\pi}} \beta_0(u) \sin(nu) \; du \]
are $0$. 
\end{proof}

In order to obtain the asymptotic behavior of the inverse curvature flow, we need to choice the center point $p$ of the scaling as stated in (ii) and (iii) in Theorem \ref{thm:asymptotics}. 
The following proposition will be used to choose the center point $p$. 

\begin{prop}
Let $(X, \nu)$ be the special inverse curvature flow satisfying \eqref{special inverse} and starting from $(X_0, \nu_0)$ obtained in Theorem \ref{thm:initial problem}. 
Then, 
\begin{equation}\label{preserve-int} 
\int_{\mathbb{S}^1_{2\pi}} X(u, t) \; du = \int_{\mathbb{S}^1_{2\pi}} X_0 (u) \; du \quad \text{for} \; \; t \ge 0. 
\end{equation}
\end{prop}

\begin{proof}
Let $(\ell, \beta)$ be the Legendre curvature of $(X, \nu)$ and $\mu:= J \nu$. 
Then, since $\ell \equiv n$ and $\nu$ is given by \eqref{construct-nu} for some $n \in \mathbb{N}$, we have 
\begin{align*} 
\partial_t \int_{\mathbb{S}^1_{2\pi}} X(u, t) \; du =&\; \int_{\mathbb{S}^1_{2\pi}} \frac{\beta(u,t)}{n} \nu(u, t) + \frac{\partial_u \beta(u, t)}{n^2} \mu(u, t) \; du \\
=&\; \int_{\mathbb{S}^1_{2\pi}} \partial_u \left( \frac{\beta(u,t)}{n^2} \mu(u, t) \right) \; du + \frac{2}{n} \int_{\mathbb{S}^1_{2\pi}} \beta (u, t) \nu(u,t) \; du \\
=&\; \frac{2}{n} J^{-1} \left( \int_{\mathbb{S}^1_{2\pi}} \beta(u, t) \mu(u,t) \; du \right) \\
=&\; \frac{2}{n} J^{-1} \left(X(2\pi, t) - X(0, t)\right) = 0 \quad \text{for} \; \; t > 0, 
\end{align*}
which yields \eqref{preserve-int}. 
\end{proof}

Let us prove Theorem \ref{thm:asymptotics}.

\begin{proof}[Proof of Theorem \ref{thm:asymptotics}]
The statement (i) follows from Proposition \ref{prop:fourier-coefficient}. 
We only prove the statement (iii) since the statement (ii) can be proved by a similar argument. 
Therefore, we assume $a_0 = \frac{1}{2\pi}\int_{\mathbb{S}^1_{2\pi}} \beta_0(u) \; du =0$ and there exists $m \in \mathbb{N}$ satisfying \eqref{as-fourier-coe}. 
Let $(\ell, \beta)$ be the Legendre curvature of $(X, \nu)$. 
Due to Proposition \ref{prop:existence-system}, $\beta$ is formed by 
\[ \beta(u, t) = \sum_{k=m}^\infty e^{(1-\frac{k^2}{n^2})t}(a_k \cos(ku) + b_k \sin(ku)) \quad \text{for} \; \; (u, t) \in \mathbb{S}^1_{2\pi} \times [0, \infty). \]
Let $(\tilde{X}, \tilde{\nu}) = (\lambda^*_{n, m}(t) X^*_{n,m,a_m, b_m}(u), \nu^*_n)$ be the self similar solution obtained in Theorem \ref{thm:self similar}. 
Then, the Legendre curvature $(\tilde{\ell}, \tilde{\beta})$ of $(\tilde{X}, \tilde{\nu})$ satisfies 
\[ \tilde{\ell}(u,t) = n, \; \; \tilde{\beta}(u, t) = \lambda^*_{n, m}(t) \beta^*_{n, m, a_m, b_m}(u) = e^{(1-\frac{m^2}{n^2})t} (a_m \cos(mu) + b_m \sin(mu)) \]
for $(u, t) \in \mathbb{S}^1_{2\pi} \times [0, \infty)$. 
Therefore, the Parseval's identity (cf.\ \cite[Section 16]{We}) yields
\[ |a_k| + |b_k| \le 2\left(\sum_{k=m}^\infty (a_k^2 + b_k^2)\right)^{1/2} = \frac{2}{\sqrt{\pi}} \|\beta_0\|_{L^2(\mathbb{S}^1_{2\pi})} = \frac{2}{\sqrt{\pi}} \left(\int_{\mathbb{S}^1_{2\pi}} (\beta_0(u))^2 \; du\right)^{1/2} \]
for $k \ge m,  \varepsilon \in (0, 1)$ and $\delta > 0$, and thus there exists $M_0 > 0$ such that 
\begin{align*} 
e^{-(1-\frac{m^2 + \varepsilon}{n^2})t} |\beta(u, t) - \tilde{\beta}(u, t)| =&\; \left|\sum_{k=m+1}^\infty e^{-\frac{k^2 - m^2 - \varepsilon}{n^2}t}(a_k \cos(ku) + b_k \sin(ku))\right| \\
\le&\; \frac{2\|\beta_0\|_{L^2(\mathbb{S}^1_{2\pi})}}{\sqrt{\pi}} \sum_{k=m+1}^\infty e^{-\frac{k^2 - m^2 - \varepsilon}{n^2}\delta} \le M_0
\end{align*}
for $(u, t) \in \mathbb{S}^1_{2\pi} \times [\delta, \infty)$. 
We thus obtain 
\begin{equation}\label{decay-beta}
\|\beta(\cdot, t) - \tilde{\beta}(\cdot, t)\|_{L^\infty(\mathbb{S}^1_{2\pi})} \le M_0 e^{(1-\frac{m^2+\varepsilon}{n^2})t} \quad \text{for} \; \; t \ge \delta. 
\end{equation}
A similar argument also yields that for any $i \in \mathbb{N}$ there exists $M_i > 0$ such that 
\begin{equation}\label{decay-beta-i}
\|\partial_u^i \beta(\cdot, t) - \partial_u^i \tilde{\beta}(\cdot, t)\|_{L^\infty(\mathbb{S}^1_{2\pi})} \le M_i e^{(1-\frac{m^2+\varepsilon}{n^2})t} \quad \text{for} \; \; t \ge \delta. 
\end{equation}

We here choose a point $p \in \mathbb{R}^2$ so that 
\[ p = \frac{1}{2\pi} \int_{\mathbb{S}^1_{2\pi}} X_0(u) \; du \]
and define 
\[ \hat{X}(u,t) := \frac{X(u, t) - p - \tilde{X}(u, t)}{\lambda^*(t)} = \frac{X(u,t) - p}{\lambda^*(t)} -X^*(u) \quad \text{for} \; \; (u, t) \in \mathbb{S}^1_{2\pi} \times [0, \infty). \]
Then, from $\int_{\mathbb{S}^1_{2\pi}} X^*(u) \; du = 0$, we have by \eqref{preserve-int}
\[ \int_{\mathbb{S}^1_{2\pi}} \hat{X}(u, t) \; du = 0 \quad \text{for} \; \; t \ge 0. \]
Therefore, the Poincar\'e inequality (cf.\ \cite[Section 5.8]{Ev}) and the Sobolev inequality (cf.\ \cite[Section 5.6]{Ev}) implies 
\begin{align*} 
\sup_{u \in \mathbb{S}^1_{2\pi}} |\hat{X}(u, t)| =:&\; \|\hat{X}(\cdot, t)\|_{L^\infty(\mathbb{S}^1_{2\pi})} \le C \|\partial_u \hat{X}(\cdot, t)\|_{L^2(\mathbb{S}^1_{2\pi})} \\
=&\; C \left(\int_{\mathbb{S}^1_{2\pi}} \frac{|\beta(u, t) - \tilde{\beta}(u, t)|^2}{(\lambda^*(t))^2} \; du \right) \le 2\pi C M_0 e^{-\frac{\varepsilon}{n^2} t}
\end{align*}
for $t \ge \delta$ and some $C > 0$. 
We here have used the estimate \eqref{decay-beta} and $\lambda^*(t) = e^{(1-\frac{m^2}{n^2})t}$ on the final inequality. 
Thus, the scaled curve $\frac{X(\cdot, t)-p}{\lambda^*(t)}$ converges to $X^*$ in $L^\infty(\mathbb{S}^1_{2\pi})$ as $t \to \infty$. 
The higher order estimate \eqref{decay-beta-i} yields the convergence in $C^i(\mathbb{S}^1_{2\pi})$ for any $i \in \mathbb{N}$, and thus it completes the proof. 
\end{proof}

\appendix

\section{Re-parametrization}

Although it is well-known that $\ell$-convex Legendre curves can be re-parametrized so that $\ell$ is a constant function, we give a proof of it for the convenience of the reader.

\begin{prop}\label{prop:re-para}
Let $\alpha \in (0, 1)$ and $(X, \nu) \in C^{1+\alpha} (\mathbb{S}^1_{2\pi}; \mathbb{R}^2 \times \mathbb{S}^1)$ be $\ell$-convex Legendre curve. 
Then, there exists an integer $n \in \mathbb{N}$ and a $C^1$-diffeomorphism $\phi \in C^{1+\alpha}(\mathbb{S}^1_{2\pi}; \mathbb{S}^1_{2\pi})$ satisfying $\partial_u \phi (u) > 0$ for $u \in \mathbb{S}^1_{2\pi}$ such that the re-parametrized Legendre curve $(\tilde{X}, \tilde{\nu}) := (X \circ \phi, \nu \circ \phi)$ satisfies 
\begin{equation}\label{re-para-nu} 
\tilde{\nu}(u) = \begin{pmatrix}
\sin (nu) \\
-\cos (nu) 
\end{pmatrix} \quad \text{for} \; \; u \in \mathbb{S}^1_{2\pi}. 
\end{equation}
Furthermore, the Legendre curvature $(\tilde{\ell}, \tilde{\beta})$ of $(\tilde{X}, \tilde{\nu})$ satisfies 
\begin{equation}\label{re-para-l} 
\tilde{\ell}(u) = n \quad \text{for} \; \; u \in \mathbb{S}^1_{2\pi}. 
\end{equation}
\end{prop}

\begin{proof}
Let $(\ell, \beta)$ be the Legendre curvature of $(X, \nu)$. 
We first choose an integer $n$. 
Let $\Theta \in C^{1+\alpha} ([0, 2\pi])$ be a function satisfying 
\[ \nu(u) = \begin{pmatrix}
\sin \Theta(u) \\
-\cos \Theta(u)
\end{pmatrix} \quad \text{for} \; \; u \in [0, 2\pi]. \]
Due to the periodicity of $\nu$, there exists an integer $n' \in \mathbb{N}$ such that 
\[ \Theta(2\pi) - \Theta(0) = 2n' \pi \]
and the integer $n$ is chosen as this integer $n'$. 
Since $\partial_u \Theta = \ell$ follows from a similar argument to obtain \eqref{eq-theta-u}, it holds that $\int_{\mathbb{S}^1} \ell(u) \; du = 2n \pi$. 

We next construct a re-parametrized Legendre curve $(\tilde{X}_1, \tilde{\nu}_1)$ satisfying $\tilde{\ell}_1(u) = n$ for $u \in \mathbb{S}^1$, where $(\tilde{\ell}_1, \tilde{\beta}_1)$ is the Legendre curvature of $(\tilde{X}_1, \tilde{\nu}_1)$. 
Define $\psi_1 \in C^{1+\alpha}(\mathbb{S}^1_{2\pi}; \mathbb{S}^1_{2\pi})$ by 
\[ \psi_1(v) := \frac{1}{n} \int_0^{v} \ell(v') \; dv' \quad \text{for} \; \; v \in \mathbb{S}^1. \]
We note that, due to the $\ell$-convexity of $(X, \nu)$, it holds that $\partial_v \psi_1(v) > 0$ for $v \in \mathbb{S}^1_{2\pi}$ and $\psi_1(2\pi) - \psi_1(0) = \frac{1}{n} \int_{\mathbb{S}^1_{2\pi}} \ell(v) \; dv = 2\pi$, which yields that $\psi_1: \mathbb{S}^1_{2\pi} \to \mathbb{S}^1_{2\pi}$ a is $C^1$-diffeomorphism. 
Therefore, the inverse function $\phi_1 \in C^{1+\alpha}(\mathbb{S}^1_{2\pi}; \mathbb{S}^1_{2\pi})$ of $\psi_1$ is well-defined and its derivative is 
\begin{equation}\label{posi-deri} 
\partial_u \phi_1(u) = \frac{n}{\ell\circ \phi_1(u)} \quad \text{for} \; \; u \in \mathbb{S}^1_{2\pi}. 
\end{equation}
Thus, letting $(\tilde{X}_1, \tilde{\nu}_1) := (X \circ \phi_1, \nu \circ \phi_1)$, we can see that 
\begin{equation}\label{re-para-l1} 
\tilde{\ell}_1(u) = \partial_u \phi_1(u) \ell\circ \phi_1(u) = n \quad \text{for} \; \; u \in \mathbb{S}^1_{2\pi}. 
\end{equation}
Notice that $\partial_u \tilde{\Theta}_1 = \tilde{\ell}_1$ yields 
\begin{equation}\label{re-para-nu1} 
\tilde{\nu}_1(u) = \begin{pmatrix}
\cos(nu + \theta_0) \\
-\sin(nu + \theta_0)
\end{pmatrix} \quad \text{for} \; \; u \in \mathbb{S}^1_{2\pi} 
\end{equation}
for some $\theta_0 \in \mathbb{R}$. 

We finally construct the function $\phi$ to define the re-parametrized Legendre curve $(\tilde{X}, \tilde{\nu})$. 
Define a $C^1$-diffeomorphism $\phi_2\in C^\infty(\mathbb{S}^1_{2\pi}; \mathbb{S}^1_{2\pi})$ by $\phi_2(u) := u - \frac{\theta_0}{n}$ and let $\phi := \phi_1 \circ \phi_2$. 
Then, from $(\tilde{X}, \tilde{\nu}) = (\tilde{X}_1 \circ \phi_2, \tilde{\nu}_1 \circ \phi_2)$, the properties \eqref{re-para-nu} and \eqref{re-para-l} follows from \eqref{re-para-l1} and \eqref{re-para-nu1}. 
The positivity of $\partial_u \phi$ follows from \eqref{posi-deri} and the $\ell$-convexity of $(X, \nu)$. 
Since $\phi_1$ and $\phi_2$ are $C^1$-diffeomorphisms, $\phi$ is also a $C^1$-diffeomorphism. 
\end{proof}

\bigskip

\noindent
{\bf Data availability}: There is no conflict of interest. 


\end{document}